\documentclass[11pt,reqno]{amsart}

\usepackage[T1]{fontenc}
\usepackage{lmodern}
\usepackage{microtype}

\usepackage{amsmath,amssymb,amsthm,mathtools}
\usepackage{booktabs}
\usepackage{placeins}

\usepackage[
  backend=biber,
  style=alphabetic,
  maxbibnames=10,
  giveninits=true,
  seconds=true,
  urldate=iso
]{biblatex}
\usepackage{xcolor}
\usepackage[
   colorlinks=true, linkcolor=blue!45!black, citecolor=blue!45!black,
  urlcolor=blue!45!black ]{hyperref}
\usepackage{orcidlink}

\numberwithin{equation}{section}

\newtheorem{theorem}{Theorem}[section]
\newtheorem{proposition}[theorem]{Proposition}
\newtheorem{lemma}[theorem]{Lemma}
\newtheorem{corollary}[theorem]{Corollary}

\theoremstyle{definition}
\newtheorem{definition}[theorem]{Definition}
\newtheorem{algorithm}[theorem]{Algorithm}
\theoremstyle{remark}
\newtheorem{remark}[theorem]{Remark}

\newcommand{\R}{\mathbb R}
\newcommand{\C}{\mathbb C}
\newcommand{\Hh}{\mathbb H}
\newcommand{\Oa}{\mathbb O}
\newcommand{\RePart}{\operatorname{Re}}
\newcommand{\ImPart}{\operatorname{Im}}
\newcommand{\ReCoeff}{\operatorname{Re}_{\C}}
\newcommand{\ImCoeff}{\operatorname{Im}_{\C}}
\newcommand{\Rk}{\operatorname{R}}
\newcommand{\AltC}{\mathsf C}
\newcommand{\AltR}{\mathsf R}
\newcommand{\Mul}{\mathsf M}

\title[Quasilinear Cayley--Dickson multiplication]{Quasilinear multiplication in
the Cayley--Dickson tower}
\author[Harrison Lemley]{Harrison Lemley\,\orcidlink{0009-0004-5016-2152}}
\address{Department of Mathematics, Purdue University, West Lafayette, IN 47907,
USA}
\email{hlemley@purdue.edu}
\email{hlemley@algebraity.com}
\thanks{This work is part of the TANUKI-CD research program on Cayley--Dickson
algebras (\url{https://www.tanuki-cd.com}).}
\date{November 23, 2026}

\subjclass[2020]{Primary 17-08; Secondary 68W30, 68W40, 15A69}
\keywords{Cayley--Dickson algebras, fast multiplication, quasilinear algorithms,
bilinear complexity, hypercomplex arithmetic, tensor
rank}

\begin{document}

\begin{abstract}
Direct evaluation of the defining product of the Cayley--Dickson algebra $A_n$
of rank $N:=2^n$ requires a quadratic number of arithmetic operations. This
paper gives a uniform bilinear algorithm that computes multiplication in $A_n$
over a unital commutative ring $R$ in which $2\in R^\times$ using $O(N\log N)$
arithmetic operations and $O(N)$ auxiliary storage, given $A_n$ is defined with
doubling parameters that are units with square roots in a fixed commutative
algebra finitely generated as an $R$-module admitting a linear surjection onto
the base ring preserving the identity. These conditions are satisfied by every
Cayley--Dickson tower over $\R$ and over finite fields of odd characteristic,
where the tensor rank of multiplication in $A_n$ is at most
$(36n-69)2^{n-3}+10$ for $n\ge4$. To the author's knowledge, this is the first
uniform subquadratic algorithm for multiplication in $A_n$. In single-core
benchmarks over $\R$, an open-source implementation in the library
\texttt{fastCD} is approximately $17$ times faster than direct multiplication
at $N=1024$.
\end{abstract}

\maketitle

\section{Introduction}

The standard real Cayley--Dickson construction starts with $A_0:=\R$ and
produces the familiar algebras $A_1\cong\C$, $A_2\cong\Hh$, $A_3\cong\Oa$, and
their higher-dimensional successors. Each step doubles the underlying vector
space so that $A_n$ has dimension $N:=2^n$. The defining product contains four
products in $A_{n-1}$, and thus direct recursive evaluation has quadratic
arithmetic complexity in $N$, requiring $N^2$ real multiplications. Standard
accounts of the construction and its low-dimensional members may be found in
\cite{Schafer1954,Baez2002}.

The quadratic cost of this multiplication has become a matter of practical
relevance in recent years. Cao, Yan, and Tan's \emph{Numerion} constructs
neural-network layers over Cayley--Dickson spaces and identifies the time and
memory required for higher-dimensional multiplication as a source of
computational overhead in their implementation \cite{CaoYanTan2026}.
Sedenion-valued models have also appeared in forecasting and traffic prediction,
while sedenion moments have been used in multiview image processing
\cite{SaoudAlMarzouqi2020,BojesomoLiatsisAlMarzouqi2020,ZhangEtAl2023}. In
hyperspectral image quality assessment, the $Q_{2^n}$ index of Garzelli and
Nencini \cite{GarzelliNencini2009} is based on second-order statistics computed
using the Cayley--Dickson product, and the computational cost of evaluating this
index for images with many spectral bands has led researchers to substitute
simpler bandwise quality indices in its place
\cite{VivoneEtAl2023,ZiniEtAl2024}. These applications illustrate a historical
pattern of practical relevance of the Cayley--Dickson tower and a growing demand
for computations within it, indicating the value of more efficient
multiplication algorithms.

A long history of work towards improving the efficiency of Cayley--Dickson
multiplication can be found in the literature. The signed-XOR description of
basis multiplication \cite{Bales2009,RenZhao2023} supports table-free methods
\cite{ParkKang2026}, while earlier hypercomplex multiplication schemes use
Walsh--Hadamard factorizations \cite{RososhekLitvinChernyaeva2000}.
Fixed-dimensional results are available for the octonions
\cite{CariowCariowa2012,CariowCariowa2013Erratum}, sedenions
\cite{CariowCariowa2013}, and the 32-dimensional Cayley--Dickson algebra
\cite{CariowCariowa2014}, with further octonion tensor-rank bounds given in
\cite{Jain2026}. The uniform Cariow--Cariowa method uses $N(N-1)/2+2$ real
multiplications together with a quadratic number of additions
\cite{CariowCariowa2015}. Work aiming to accelerate high-dimensional
Cayley--Dickson multiplication through GPU parallelism has also been reported
\cite{YottaSpaceStatistics2026, PerdueCayleyKernel2026}, but such parallel
acceleration does not by itself improve the arithmetic complexity of the
underlying multiplication algorithm. To the author's knowledge, a uniform
subquadratic algorithm for multiplying elements of the standard real
Cayley--Dickson tower has never appeared in the literature, and the best
previously known uniform upper bound for the tensor rank of multiplication in
$A_n$ is quadratic in $N$. The question of a uniform subquadratic algorithm was
posed explicitly on MathOverflow by Kulkov in 2023 \cite{Kulkov2023}.

Vector-matrix descriptions of octonionic and related algebras originate with
Zorn, with subsequent developments by Paige and others
\cite{Zorn1931,Paige1957,Suh1969,DaboulDelbourgo1999,Gazeau2026}. The most
direct structural precedent here is the Brown--Hopkins noncommutative matrix
Jordan construction \cite{BrownHopkins1992}, which they iterate using the
trace-zero subspace with its commutator product and trace pairing. They give
explicit algebra isomorphisms relating this iteration to the Cayley--Dickson
process over $\R$ \cite[Lemma~3.17, Lemma~4.4, and
Theorem~4.6]{BrownHopkins1995}. After scalar extension to $\C$, Lemma~3.17
identifies the standard Cayley--Dickson doubling parameter $-1$ with their
normalized parameter $1$. Thus, up to the corresponding coordinate changes and
normalizations of the commutator and bilinear form, the recursive algebraic
structure used in this paper is classical and is not claimed as a contribution
of the present work.

The Cayley--Dickson construction need not start with $\R$. The procedure can
start from a unital commutative ring $R$, with a fixed doubling parameter
$\mu_k\in R^\times$ for each $k\in\mathbb Z^{>0}$
\cite[Section~6]{McCrimmon1985}. Here $A_n$ is a free $R$-module of rank
$N=2^n$, and direct evaluation of the defining recursive product still uses a
quadratic number of arithmetic operations in the underlying ring $R$.

The main contribution of this paper is a uniform quasilinear algorithm obtained
by reorganizing the classical recursive structure. As in Karatsuba--Ofman
multiplication \cite{KaratsubaOfman1962}, the improvement comes from reducing
the number of recursive calls, in this case by passing to the alternating
product on imaginary elements. After extending scalars to one fixed commutative
$R$-algebra $S$, finitely generated as an $R$-module and containing the required
square roots, each recursive step uses two half-size recursive calls and
linear-time recombination. The scalar--vector products and weighted pairings are
evaluated jointly through a nonrecursive bilinear map, and an $R$-linear map
recovers $R$-valued coefficients at the root. The same coefficient algebra is
used throughout, and full multiplication is recovered by scalar arithmetic and
weighted bilinear pairings. For the standard real tower, $S=\C$, and on real
inputs only one initial recursive call is required
(Corollary~\ref{cor:one-call}).

\begin{theorem}\label{thm:intro-main}
Let $R$ be a unital commutative ring with $2\in R^\times$, and fix
$\mu_k\in R^\times$ for each $k\in\mathbb Z^{>0}$. Suppose that a fixed
commutative $R$-algebra $S$, finitely generated as an $R$-module, contains
elements $\rho_k$ satisfying $\rho_k^2=\mu_k$ for $k\ge3$ and admits an
$R$-linear map $\pi\colon S\to R$ with $\pi(1)=1$. For $n\ge1$, let $A_n$ be the
corresponding Cayley--Dickson algebra, free of rank $N=2^n$ over $R$. Then the
product of two elements $x, y \in A_n$ can be computed using
\[
O(n2^n)=O(N\log N)
\]
arithmetic operations in $R$, with a sequential evaluation using $O(N)$
auxiliary $R$-coefficients. The algorithm is bilinear, so that every
multiplication of two input-dependent quantities is the product of a fixed
$R$-linear form in the coordinates of $x$ with a fixed $R$-linear form in the
coordinates of $y$, with at most $O(N\log N)$ such multiplications used.
\end{theorem}

The general arithmetic, bilinear, and storage bounds in
Theorem~\ref{thm:intro-main} follow from Corollaries~\ref{cor:quasilinear}
and~\ref{thm:rank} and Proposition~\ref{prop:workspace} in
Section~\ref{sec:complexity}. Since the algorithm is bilinear,
Theorem~\ref{thm:intro-main} also bounds the bilinear complexity of
multiplication in $A_n$, its tensor rank when $R$ is a field, by $O(N\log N)$.
Corollary~\ref{cor:field-real} applies this over $\R$ and over finite fields of
odd characteristic, where the bound becomes $(9n-15)2^{n-1}+10$ for every
$n\ge1$. Theorem~\ref{thm:leaf-rank} improves this to $(36n-69)2^{n-3}+10$ for
$n\ge4$ by using an octonion base case, and Appendix~\ref{sec:real-schedule}
gives the exact operation counts of the resulting schedule for the standard real
tower, verified by an accompanying script.

The accompanying \texttt{fastCD} library, with the Python package
\texttt{pyfastcd} available on PyPI, provides C and NumPy-backed Python
interfaces \cite{HarrisEtAl2020} for multiplication in $A_n$ over $\R$, $\C$,
and $\mathbb{F}_q$ for $q = p^k$ with $p \le 2^{31}-1$ an odd prime and
$k \le 12$, and comparisons with direct and uniform Cariow--Cariowa
multiplication are presented in Section~\ref{sec:implementation}. In controlled
benchmarks, quasilinear multiplication had the lowest mean time among the three
methods at every tested dimension $N\ge32$ for single and batched products over
$\R$, while direct and Cariow--Cariowa multiplication were faster at some lower
dimensions. The section also reports benchmarks over $\C$ and selected finite
fields and measurements of floating-point error.

Sections~\ref{sec:preliminaries} and~\ref{sec:recursion} derive the recursion
directly from the Cayley--Dickson product; Section~\ref{sec:complexity} gives
explicit arithmetic bounds and proves the general quasilinear, bilinear, and
storage bounds; Section~\ref{sec:implementation} presents the \texttt{fastCD}
implementation and benchmarks; finally, Section~\ref{sec:outlook} suggests some
questions for future work and Section~\ref{sec:disclosures} gives statements and
disclosures. Appendix~\ref{sec:real-schedule} gives the octonion base case and
exact operation counts of the specialized schedule with an accompanying
verification script.

\section{Scalar and imaginary parts of Cayley--Dickson multiplication}
\label{sec:preliminaries}

Let $R$ be a commutative ring with identity in which $2$ is invertible, and fix
$\mu_k\in R^\times$ for each $k\in\mathbb Z^{>0}$. Set $A_0:=R$. Recursively,
\[
A_n:=A_{n-1}\oplus A_{n-1}\,,
\]
with conjugation and multiplication
\begin{equation}\label{eq:CD-product}
\overline{(a,b)}:=(\bar a,-b)\,,
\qquad
(a,b)(c,d):=\bigl(ac+\mu_n\bar d\,b,\;da+b\bar c\bigr)\,,
\end{equation}
where $\bar r=r$ for $r\in R$. This is the Cayley--Dickson construction over a
ring of scalars; see \cite[Section~6]{McCrimmon1985}. Each $A_n$ is free of rank
$N=2^n$ over $R$. The standard real tower is the case $R=\R$ and $\mu_k=-1$ for
every $k$.

Throughout this paper, use the recursively concatenated basis
$e_0,\ldots,e_{N-1}$, with $e_0=1$, such that the coordinates of $(a, b)$ are
those of $a$ followed by those of $b$. These will be referred to as the standard
coordinates of an element $(a, b) \in A_n$.

Define a quadratic form $q_n\colon A_n\to R$ recursively by
\[
q_0(r):=r^2\,,\qquad
q_n(a,b):=q_{n-1}(a)-\mu_n q_{n-1}(b)\,,
\]
and put
\begin{equation}\label{eq:pairing}
\langle x,y\rangle_n
:=\frac{q_n(x+y)-q_n(x)-q_n(y)}{2}\,.
\end{equation}
This is a symmetric $R$-bilinear form. In standard coordinates,
\begin{equation}\label{eq:weights}
q_n(x)=\sum_{j=0}^{N-1}w_{n,j}x_j^2\,,\qquad
\langle x,y\rangle_n=\sum_{j=0}^{N-1}w_{n,j}x_jy_j\,,
\end{equation}
where $w_{0,0}:=1$ and the weight lists
$\mathbf w_n:=(w_{n,0},\ldots,w_{n,N-1})$ satisfy
\[
\mathbf w_n=(\mathbf w_{n-1},-\mu_n\mathbf w_{n-1})\,.
\]
The weights depend only on the fixed parameters. The subscript on the pairing
will be omitted when the level is clear. Throughout, a scalar denotes the
corresponding scalar multiple of the unit, and these scalars are central in
every $A_n$ by induction from \eqref{eq:CD-product}.

For $x\in A_n$, define its scalar and imaginary parts by
\[
\RePart(x):=\frac{x+\bar x}{2}\,,\qquad
\ImPart(x):=\frac{x-\bar x}{2}\,,
\]
identifying the scalar part with its coefficient in $R$. Put
\[
V_n:=\ImPart(A_n)=\{x\in A_n:\bar x=-x\}\,.
\]
Then $A_n=R\oplus V_n$, and $V_n$ is free of rank $2^n-1$ over $R$.

\begin{lemma}[Standard involution identities]\label{lem:involution}
For every $n\ge0$ and $x,y\in A_n$,
\[
\overline{xy}=\bar y\,\bar x\,,\qquad
x\bar x=\bar x x=q_n(x)\,.
\]
In particular, if $u\in V_n$, then $u^2=-q_n(u)$.
\end{lemma}

\begin{proof}
The statements follow by induction from \eqref{eq:CD-product}. The scalar base
case is immediate. For the induction step, let $x=(a,b)$ and $y=(c,d)$. The
inductive hypothesis gives
\[
\begin{aligned}
\overline{xy}
&=\bigl(\bar c\,\bar a+\mu_n\bar b d,\;-da-b\bar c\bigr) \\
&=(\bar c,-d)(\bar a,-b)=\bar y\,\bar x\,.
\end{aligned}
\]
For the quadratic identity,
\[
\begin{aligned}
x\bar x
&=(a,b)(\bar a,-b) \\
&=\bigl(a\bar a-\mu_n\bar b b,0\bigr) \\
&=\bigl(q_{n-1}(a)-\mu_n q_{n-1}(b),0\bigr)\,.
\end{aligned}
\]
The calculation of $\bar x x$ is the same. If $u\in V_n$, then $\bar u=-u$, so
$u\bar u=q_n(u)$ gives $u^2=-q_n(u)$.
\end{proof}

\begin{definition}\label{def:alternating-product}
For $u,v\in V_n$, define
\begin{equation}\label{eq:cross-def}
u\times_n v:=\frac12(uv-vu)\,.
\end{equation}
\end{definition}

\begin{lemma}[Product of imaginary elements]\label{lem:pure-product}
For $u,v\in V_n$,
\begin{equation}\label{eq:pure-product}
uv=-\langle u,v\rangle_n+u\times_n v\,.
\end{equation}
Moreover, $u\times_n v\in V_n$ and $\times_n$ is $R$-bilinear and alternating.
\end{lemma}

\begin{proof}
Since $u+v\in V_n$, Lemma~\ref{lem:involution} gives
\[
(u+v)^2=-q_n(u+v)\,.
\]
Expanding and using the corresponding identities for $u$ and $v$ gives
\[
-q_n(u)+uv+vu-q_n(v)
=
-q_n(u+v)\,.
\]
Since
\[
q_n(u+v)=q_n(u)+2\langle u,v\rangle_n+q_n(v)\,,
\]
it follows that
\[
uv+vu=-2\langle u,v\rangle_n\,.
\]
Decomposing $uv$ into its symmetric and antisymmetric parts now yields
\[
uv
=
\frac12(uv+vu)+\frac12(uv-vu)\,.
\]
The preceding identity and Definition~\ref{def:alternating-product} give
\eqref{eq:pure-product}.

The $R$-bilinearity of $\times_n$ follows immediately from bilinearity of
multiplication. Moreover,
\[
v\times_n u=\frac12(vu-uv)=-u\times_n v\,,
\]
and $u\times_n u=\frac12(uu-uu)=0$, so $\times_n$ is alternating. Finally,
because $\bar u=-u$, $\bar v=-v$, and conjugation reverses products,
\[
\begin{aligned}
\overline{u\times_n v}
&=
\frac12\bigl(\overline{uv}-\overline{vu}\bigr) \\
&=
\frac12\bigl(\bar v\bar u-\bar u\bar v\bigr) \\
&=
\frac12(vu-uv) \\
&=
-u\times_n v\,.
\end{aligned}
\]
Thus $u\times_n v\in V_n$.
\end{proof}

\begin{proposition}[Scalar--imaginary decomposition]
\label{prop:scalar--imaginary}
Let $x=\alpha+X$ and $y=\beta+Y$, where $\alpha,\beta\in R$ and $X,Y\in V_n$.
Then
\begin{equation}\label{eq:scalar--imaginary}
xy=\bigl(\alpha\beta-\langle X,Y\rangle_n\bigr)
   +\alpha Y+\beta X+X\times_nY\,.
\end{equation}
\end{proposition}

\begin{proof}
Scalars are central in $A_n$, so
\[
\begin{aligned}
xy
&=(\alpha+X)(\beta+Y) \\
&=\alpha\beta+\alpha Y+\beta X+XY \\
&=\bigl(\alpha\beta-\langle X,Y\rangle_n\bigr)
  +\alpha Y+\beta X+X\times_nY\,,
\end{aligned}
\]
where the last equality follows from Lemma~\ref{lem:pure-product}.
\end{proof}

Thus, the only part of multiplication requiring a recursive Cayley--Dickson
calculation is the alternating product $\times_n$, and all remaining terms
require only $O(2^n)=O(N)$ arithmetic operations in $R$.

\section{A two-call recursion}
\label{sec:recursion}

Fix $n\ge2$ and write $B:=A_{n-1}$ and $V:=V_{n-1}$. An element
$(a,b)\in B\oplus B\cong A_n$ is imaginary exactly when $a\in V$. Its second
coordinate is unrestricted. Hence every $X\in V_n$ has a unique expression
\begin{equation}\label{eq:imaginary-splitting}
X=(u,\gamma+s)\,,\qquad u,s\in V\,,\quad\gamma\in R\,.
\end{equation}

\begin{proposition}[Doubled alternating product]\label{prop:doubled-alt}
Let $X=(u,\gamma+s)$ and $Y=(v,\delta+t)$, where $u,v,s,t\in V$ and
$\gamma,\delta\in R$. Define
\begin{equation}\label{eq:DH}
\begin{aligned}
D&:=u\times_{n-1}v+\mu_n s\times_{n-1}t\,, \\
H&:=u\times_{n-1}t+s\times_{n-1}v\,.
\end{aligned}
\end{equation}
Then
\begin{equation}\label{eq:doubled-alt}
\begin{aligned}
X\times_nY=\Bigl(&D+\mu_n(\delta s-\gamma t)\,, \\
&\bigl(\langle s,v\rangle_{n-1}-\langle t,u\rangle_{n-1}\bigr)
 +\delta u-\gamma v-H\Bigr)\,.
\end{aligned}
\end{equation}
\end{proposition}

\begin{proof}
Put $b:=\gamma+s$ and $d:=\delta+t$. By \eqref{eq:CD-product},
\[
XY=(uv+\mu_n\bar d b,\;du-bv)\,,\qquad
YX=(vu+\mu_n\bar b d,\;bv-du)\,.
\]
Therefore
\begin{equation}\label{eq:cross-expanded}
X\times_nY
=\left(\frac{uv-vu+\mu_n(\bar d b-\bar b d)}{2},\;du-bv\right)\,.
\end{equation}
Since $\bar b=\gamma-s$ and $\bar d=\delta-t$,
\[
\frac{\bar d b-\bar b d}{2}
=\delta s-\gamma t+s\times_{n-1}t\,.
\]
This gives the first coordinate in \eqref{eq:doubled-alt}. For the second,
Lemma~\ref{lem:pure-product} gives
\[
\begin{aligned}
du-bv
&=\delta u-\gamma v+tu-sv \\
&=\langle s,v\rangle_{n-1}-\langle t,u\rangle_{n-1}
  +\delta u-\gamma v-u\times_{n-1}t-s\times_{n-1}v\,.
\end{aligned}
\]
Substituting the definitions of $D$ and $H$ proves the assertion.
\end{proof}

\begin{remark}[Relation to Brown--Hopkins]\label{rem:brown--hopkins}
For the standard real tower, Proposition~\ref{prop:doubled-alt} is the standard
Cayley--Dickson coordinate form of the trace-zero commutator recursion used in
the Brown--Hopkins construction, with the commutator normalized by a factor of
$1/2$. Their construction iterates the trace-zero subspace with its commutator
product and trace pairing, and their explicit isomorphisms identify the
resulting algebras with iterated Cayley--Dickson algebras \cite[Lemma~3.17,
Lemma~4.4, and Theorem~4.6]{BrownHopkins1995} after scalar extension to $\C$.
Under these identifications, trace-zero elements correspond to imaginary
elements, the commutator is twice the alternating product in this paper, and the
trace pairing is a constant multiple of the standard bilinear pairing used here.
The calculation above fixes the signs and normalization over $R$ with the
parameters used here, while the contribution of the present paper is the uniform
quasilinear arithmetic schedule built from this classical structure.
\end{remark}

\subsection{A fixed coefficient algebra}

Let $S$ be a fixed commutative $R$-algebra, finitely generated as an $R$-module,
with an $R$-linear map $\pi\colon S\to R$ satisfying $\pi(1)=1$ and elements
$\rho_k\in S$ for which
\begin{equation}\label{eq:coefficient-hypothesis}
\rho_k^2=\mu_k\qquad(k\ge3)\,.
\end{equation}
Since each $\mu_k$ is a unit, $\rho_k$ is also a unit, with inverse
$\mu_k^{-1}\rho_k$.

Let
\[
A_n^S:=A_n\otimes_R S\,,\qquad V_n^S:=V_n\otimes_R S\,.
\]
The standard coordinates of these modules are allowed to lie in $S$, and the
Cayley--Dickson involution is extended $S$-linearly so that it fixes every
coefficient in $S$. The pairing is extended $S$-bilinearly, with the same
weights as in \eqref{eq:weights}. The identities of
Section~\ref{sec:preliminaries}, including the scalar--imaginary decomposition,
and \eqref{eq:doubled-alt} remain valid after scalar extension, as they are
identities with structure constants in $R$. Let
\[
\AltC_n\colon V_n^S\times V_n^S\longrightarrow V_n^S
\]
denote the $S$-bilinear extension of $\times_n$. Since $\pi(r1)=r$ for every
$r\in R$, the natural map $R\to S$ is injective, and applying $\pi$
coefficientwise recovers coordinates that lie in $R$.

Choose $R$-module generators $\zeta_1=1,\zeta_2,\ldots,\zeta_d$ independently of
$n$, and represent an element of $S$ by a list $(a_1,\ldots,a_d)$ with value
$\sum_i a_i \zeta_i$. Although these representations need not be unique,
choosing expressions for the products $\zeta_i \zeta_j$ in the generators gives
coordinate operations that respect the represented elements of $S$. The
multiplication table uses at most $d^2$ input-dependent $R$-products, while
$R$-linearity gives
\[
\pi\left(\sum_i a_i \zeta_i\right)=\sum_i a_i\pi(\zeta_i)\,.
\]
Thus addition, multiplication, multiplication by a fixed coefficient, and
application of $\pi$ can all be performed with a number of operations in $R$
bounded in terms of $d$, using the chosen structure constants and values
$\pi(\zeta_i)$.

\subsection{The two-call formula}

\begin{theorem}[Two-call formula]\label{thm:two-call}
Let $n\ge3$ and let $X=(u,\gamma+s)$ and $Y=(v,\delta+t)$ belong to $V_n^S$.
Define
\begin{equation}\label{eq:zpm}
Z_+:=\AltC_{n-1}(u+\rho_n s,v+\rho_n t)\,,\qquad
Z_-:=\AltC_{n-1}(u-\rho_n s,v-\rho_n t)\,,
\end{equation}
and put
\[
\lambda:=\langle s,v\rangle_{n-1}-\langle t,u\rangle_{n-1}\,.
\]
Then
\begin{equation}\label{eq:two-call-formula}
\begin{aligned}
\AltC_n(X,Y)=\Bigl(&\frac{Z_++Z_-}{2}+\mu_n(\delta s-\gamma t)\,, \\
&\lambda+\delta u-\gamma v-\frac{Z_+-Z_-}{2\rho_n}\Bigr)\,.
\end{aligned}
\end{equation}
Consequently, one evaluation of $\AltC_n$ requires two evaluations of
$\AltC_{n-1}$ together with $O(2^n)$ arithmetic operations in $S$.
\end{theorem}

\begin{proof}
Let $D,H\in V_{n-1}^S$ denote the scalar extensions of the expressions in
\eqref{eq:DH}; explicitly,
\[
\begin{aligned}
D&:=\AltC_{n-1}(u,v)+\mu_n\AltC_{n-1}(s,t)\,, \\
H&:=\AltC_{n-1}(u,t)+\AltC_{n-1}(s,v)\,.
\end{aligned}
\]
Since $u,s\in V_{n-1}^S$ and the Cayley--Dickson involution is extended
$S$-linearly,
\[
\bar u=-u\,,\qquad\bar s=-s\,,
\]
and hence
\[
\overline{u\pm\rho_n s}
=\bar u\pm\rho_n\bar s
=-u\mp\rho_n s
=-(u\pm\rho_n s)\,.
\]
Thus $u\pm\rho_n s\in V_{n-1}^S$, and likewise $v\pm\rho_n t\in V_{n-1}^S$, so
both recursive calls in \eqref{eq:zpm} are well defined.

By $S$-bilinearity,
\[
\begin{aligned}
Z_\pm
&=\AltC_{n-1}(u\pm\rho_n s,v\pm\rho_n t) \\
&=\AltC_{n-1}(u,v)
 \pm\rho_n\AltC_{n-1}(u,t)
 \pm\rho_n\AltC_{n-1}(s,v)
 +\rho_n^2\AltC_{n-1}(s,t) \\
&=D\pm\rho_n H\,,
\end{aligned}
\]
where $\rho_n^2=\mu_n$ was used in the last line. Since $2$ and $\rho_n$ are
units,
\[
D=\frac{Z_++Z_-}{2}\,,\qquad H=\frac{Z_+-Z_-}{2\rho_n}\,.
\]
Substituting these identities into the scalar extension of
\eqref{eq:doubled-alt} gives
\[
\begin{aligned}
\AltC_n(X,Y)=\Bigl(&\frac{Z_++Z_-}{2}+\mu_n(\delta s-\gamma t)\,, \\
&\lambda+\delta u-\gamma v-\frac{Z_+-Z_-}{2\rho_n}\Bigr)\,,
\end{aligned}
\]
which is \eqref{eq:two-call-formula}. All remaining work consists of
scalar--vector products, two weighted bilinear pairings, and linear
recombination, and therefore requires $O(2^n)$ arithmetic operations in $S$.
\end{proof}

The module $V_1^S$ is free of rank one over $S$, and $\AltC_1\equiv0$, forming
the first base case of the recursion. At the second level, the preceding-level
alternating products vanish, so \eqref{eq:doubled-alt} gives an explicit
$6$-product formula.

The number of recursive calls required at the root can be reduced to one when
the following two projection maps exist.

\begin{corollary}[One-call root]\label{cor:one-call}
Let $n\ge3$, and suppose that $R$-linear maps $\pi_0,\pi_1\colon S\to R$ satisfy
$\pi_0(r+\rho_n r')=r$ and $\pi_1(r+\rho_n r')=r'$ for all $r,r'\in R$. Let
$X=(u,\gamma+s)$ and $Y=(v,\delta+t)$ belong to $V_n$, put
$Z:=\AltC_{n-1}(u+\rho_n s,v+\rho_n t)$, and let $\lambda$ be as in
Theorem~\ref{thm:two-call}. Then, with $\pi_0$ and $\pi_1$ applied
coefficientwise,
\[
X\times_nY=\bigl(\pi_0(Z)+\mu_n(\delta s-\gamma t),\;
\lambda+\delta u-\gamma v-\pi_1(Z)\bigr)\,.
\]
\end{corollary}

\begin{proof}
Since $u,v,s,t$ have coefficients in $R$, so do $D$ and $H$ in \eqref{eq:DH}.
The computation in the proof of Theorem~\ref{thm:two-call} gives $Z=D+\rho_n H$,
so every coefficient of $Z$ has the form $d+\rho_n h$ with $d,h\in R$, and
$\pi_0$ and $\pi_1$ return $D$ and $H$. Substituting into \eqref{eq:doubled-alt}
gives the formula.
\end{proof}

Although the reduction requires $R$-valued operands, a two-call step with
$\rho_n\in R$ preserves this property, so the reduction can also be used at
lower levels when the required projection maps are available. Once the recursion
enters the kernel over $S$, its operands need not have coefficients in $R$.

\subsection{The recursive algorithm}

The scalar--vector products and pairings outside the recursive calls in
\eqref{eq:doubled-alt} are bilinear in $X$ and $Y$, so they can be evaluated
together before being combined with the recursive outputs. Using the coordinates
$(u,\gamma,s)$ from \eqref{eq:imaginary-splitting}, define the nonrecursive
$R$-bilinear map
\begin{equation}\label{eq:nonrecursive-map}
\begin{aligned}
B_n\bigl((u,\gamma,s),(v,\delta,t)\bigr)
:=\bigl(&\mu_n(\delta s-\gamma t),\;
\langle s,v\rangle_{n-1}-\langle t,u\rangle_{n-1}\,, \\
&\delta u-\gamma v\bigr)\,.
\end{aligned}
\end{equation}
Fix a bilinear formula on the chosen coordinate representations of $S$ for the
map $(a,b)\mapsto\theta(a)\theta(b)$, where $\theta(a) := \sum_i a_i \zeta_i$,
and let $\beta_S$ denote its number of input-dependent $R$-products. Such a
formula can always be chosen with $1\le\beta_S\le d^2$. Write coefficient
multiplication in coordinates as
\begin{equation}\label{eq:coefficient-decomposition}
\theta(a)\theta(b)
=\sum_{\ell=1}^{\beta_S}
\alpha_\ell(a)\beta_\ell(b)c_\ell\,,
\end{equation}
where the forms $\alpha_\ell,\beta_\ell\colon R^d\to R$ are $R$-linear and the
$c_\ell\in S$ have fixed coordinate representations. For coordinate
representations $\mathbf X,\mathbf Y$ of $X,Y\in V_n^S$, evaluate
\begin{equation}\label{eq:general-fusion}
B_n^S(X,Y)
=\sum_{\ell=1}^{\beta_S}c_\ell
B_n\bigl(\alpha_\ell(\mathbf X),\beta_\ell(\mathbf Y)\bigr)\,.
\end{equation}
Here $B_n^S$ is the scalar extension of $B_n$, and each linear form is applied
to every input coefficient representation. Substituting
\eqref{eq:coefficient-decomposition} into the coordinate products defining $B_n$
and collecting the terms for each $\ell$ gives \eqref{eq:general-fusion}. This
evaluates the nonrecursive map by $\beta_S$ calls over $R$, with coefficient
reconstruction performed once for each output coordinate.

Algorithm~\ref{alg:multiplication} gives the recursive kernel for evaluating
$\AltC_n$ over $S$, its specialization to inputs with coefficients in $R$, and
full multiplication. For $R$-valued inputs, the root applies $\pi$
coefficientwise to recover the alternating product over $R$, or uses the maps
$\pi_0$ and $\pi_1$ of Corollary~\ref{cor:one-call}. The branch of $\AltR_n$
using the latter maps is called the \emph{one-call root}. The level-two formula
supplies the general base case.

\begin{algorithm}[Recursive multiplication]\label{alg:multiplication}
For $n\ge1$ and $X,Y\in V_n^S$, compute $\AltC_n(X,Y)$ by
\begin{tabbing}
\qquad\=\qquad\=\kill
\>\textbf{if} $n=1$, \textbf{return} $0$; \\
\>\textbf{if} a specified higher base case applies,
  \textbf{return} its value; \\
\>write $X=(u,\gamma+s)$ and $Y=(v,\delta+t)$; \\
\>$(\mathcal F,\lambda,\mathcal G)\leftarrow B_n^S(X,Y)$ by
\eqref{eq:general-fusion}; \\
\>\textbf{if} $n=2$, \textbf{return} $(\mathcal F,\lambda+\mathcal G)$; \\
\>$Z_+\leftarrow\AltC_{n-1}(u+\rho_n s,v+\rho_n t)$; \\
\>$Z_-\leftarrow\AltC_{n-1}(u-\rho_n s,v-\rho_n t)$; \\
\>\textbf{return} $\bigl((Z_++Z_-)/2+\mathcal F,
  \;\lambda+\mathcal G-(Z_+-Z_-)/(2\rho_n)\bigr)$.
\end{tabbing}
\par\noindent\begin{minipage}{\linewidth}
For $X,Y\in V_n$, compute $\AltR_n(X,Y)$ by
\begin{tabbing}
\qquad\=\qquad\=\kill
\>\textbf{if} $n=1$, \textbf{return} $0$; \\
\>write $X=(u,\gamma+s)$ and $Y=(v,\delta+t)$; \\
\>$(\mathcal F,\lambda,\mathcal G)\leftarrow B_n(X,Y)$; \\
\>\textbf{if} $n=2$, \textbf{return} $(\mathcal F,\lambda+\mathcal G)$; \\
\>\textbf{if} $\pi_0,\pi_1$ as in Corollary~\ref{cor:one-call} are given: \\
\>\>$Z\leftarrow\AltC_{n-1}(u+\rho_n s,v+\rho_n t)$; \\
\>\>$D\leftarrow\pi_0(Z)$,\quad $H\leftarrow\pi_1(Z)$; \\
\>\textbf{else if} $\rho_n\in R$: \\
\>\>$Z_+\leftarrow\AltR_{n-1}(u+\rho_n s,v+\rho_n t)$; \\
\>\>$Z_-\leftarrow\AltR_{n-1}(u-\rho_n s,v-\rho_n t)$; \\
\>\>$D\leftarrow(Z_++Z_-)/2$,\quad $H\leftarrow(Z_+-Z_-)/(2\rho_n)$; \\
\>\textbf{else}: \\
\>\>$Z_+\leftarrow\AltC_{n-1}(u+\rho_n s,v+\rho_n t)$; \\
\>\>$Z_-\leftarrow\AltC_{n-1}(u-\rho_n s,v-\rho_n t)$; \\
\>\>$D\leftarrow\pi\bigl((Z_++Z_-)/2\bigr)$; \\
\>\>$H\leftarrow\pi\bigl((Z_+-Z_-)/(2\rho_n)\bigr)$; \\
\>\textbf{return} $(D+\mathcal F,\lambda+\mathcal G-H)$.
\end{tabbing}
\end{minipage}\par
\par\noindent\begin{minipage}{\linewidth}
Finally, for $n\ge0$ and $x,y\in A_n$, compute $\Mul_n(x,y)$ by
\begin{tabbing}
\qquad\=\qquad\=\kill
\>\textbf{if} $n=0$, \textbf{return} the product $xy$ in $R$; \\
\>$\alpha\leftarrow\RePart(x)$,\quad $X\leftarrow\ImPart(x)$; \\
\>$\beta\leftarrow\RePart(y)$,\quad $Y\leftarrow\ImPart(y)$; \\
\>$Z\leftarrow\AltR_n(X,Y)$; \\
\>\textbf{return} $(\alpha\beta-\langle X,Y\rangle_n)+\alpha Y+\beta X+Z$.
\end{tabbing}
\end{minipage}\par
\end{algorithm}

\begin{proposition}[Correctness]\label{prop:correctness}
For every $n\ge1$, the recursive kernel computes the $S$-bilinear extension of
$\times_n$, and $\AltR_n(X,Y)=X\times_nY$ for $X,Y\in V_n$. Consequently,
$\Mul_n(x,y)=xy$ for all $n\ge0$ and $x,y\in A_n$.
\end{proposition}

\begin{proof}
For the recursive kernel, induction starts with $\AltC_1\equiv0$ and the
explicit level-two formula, and proceeds by Theorem~\ref{thm:two-call}, with
\eqref{eq:general-fusion} evaluating the nonrecursive map. An exact base case at
a fixed higher level gives the same induction.

For $\AltR_n$, use induction on $n$ with the same base cases. In the branch
$\rho_n\in R$, the transformed inputs lie in $V_{n-1}$, so the induction
hypothesis identifies the two recursive outputs with $Z_+$ and $Z_-$ in
Theorem~\ref{thm:two-call}, and the two-call identity gives $X\times_nY$. In the
remaining two-call branch, the quantities $D,H$ in \eqref{eq:DH} have
coefficients in $R$, on which $\pi$ is the identity, so projection recovers the
same result. Correctness of the one-call branch follows from
Corollary~\ref{cor:one-call}. Formula~\eqref{eq:scalar--imaginary} recovers full
multiplication, and the case $n=0$ is the scalar base case.
\end{proof}

\section{Arithmetic complexity and tensor rank}
\label{sec:complexity}

\subsection{Arithmetic complexity}

It remains to determine the arithmetic complexity of the algorithm of the
previous section. In this section, additions, subtractions, and multiplications
in $R$ are counted in the unit-cost arithmetic model, including multiplication
by fixed scalars. Scalar and imaginary parts are read from the standard
coordinates. The coefficient multiplication table, the values of the projection
maps on the chosen generators, and the coordinates of $\rho_k$ and
$(2\rho_k)^{-1}$ are fixed precomputed data, so the cost of computing them is
not counted. The projection maps are $\pi$ and, when the one-call root is used,
$\pi_0$ and $\pi_1$. The count concerns evaluation on the variable inputs using
these data and includes preparation of the pairing weights. All logarithms in
complexity estimates are to base $2$.

Write $M_{\mathrm{gen}}$ for the number of multiplications of two
input-dependent quantities, $A_{\mathrm{gen}}$ for additions and subtractions,
and $E_{\mathrm{gen}}$ for multiplications by fixed scalars in a full
multiplication. The total number of operations performed by the algorithm is
$M_{\mathrm{gen}}+A_{\mathrm{gen}}+E_{\mathrm{gen}}$. The bounds below use the
chosen $d$ generators of $S$ and the coefficient formula
\eqref{eq:coefficient-decomposition}, whose product count is $\beta_S$.

\begin{theorem}[Arithmetic counts]\label{thm:arithmetic}
Under the coefficient-algebra hypothesis of Section~\ref{sec:recursion}, put
\[
\begin{aligned}
\kappa_A&:=3d^3+6d\beta_S+3d\,, \\
\kappa_E&:=3d^3+3d^2+d+(6d+3)\beta_S\,.
\end{aligned}
\]
Algorithm~\ref{alg:multiplication}, using the level-two base case for the
recursive kernel, admits an evaluation with the following arithmetic bounds for
$n\ge1$:
\begin{equation}\label{eq:general-arithmetic-counts}
\begin{aligned}
M_{\mathrm{gen}}
&\le\bigl(3\beta_S(n-3)+6\bigr)2^n+12\beta_S-8\,, \\
A_{\mathrm{gen}}
&\le\bigl(\kappa_A(n-1)+8\bigr)2^{n-1}-4\,, \\
E_{\mathrm{gen}}
&\le\bigl(\kappa_E(n-1)+4\bigr)2^{n-1}-2\,.
\end{aligned}
\end{equation}
For $n=1$, the respective counts are at most $(4,3,2)$, including weight
preparation. For $n=0$, $1$ multiplication in $R$ suffices.

With the one-call root for $n\ge3$ and the direct base cases for $n=1,2$, the
multiplication count satisfies
\begin{equation}\label{eq:one-call-counts}
M_{\mathrm{gen}}\le\bigl(3\beta_S(n-3)+12\bigr)2^{n-1}+6\beta_S-8\,,
\end{equation}
and the bounds on $A_{\mathrm{gen}}$ and $E_{\mathrm{gen}}$ in
\eqref{eq:general-arithmetic-counts} remain valid.
\end{theorem}

\begin{proof}
Record costs as triples $(M,A,E)$, counting input-dependent multiplications,
additions and subtractions, and multiplications by fixed scalars respectively,
comparing and adding triples componentwise. For $n\ge2$, the evaluation consists
of the call to $\AltR_n$ and its recursive calls, the reconstruction
\eqref{eq:scalar--imaginary} in $\Mul_n$, and preparation of the pairing
weights. The proof bounds the work that a single call performs outside its
recursive calls, counts the calls at each level, and then adds the
reconstruction and weight preparation.

\emph{Arithmetic in $S$.} Write the chosen multiplication table as
\[
\zeta_i\zeta_j=\sum_{h=1}^d t_{ijh}\zeta_h\,,\qquad t_{ijh}\in R\,.
\]
For a fixed $c=\sum_j c_j\zeta_j$ and $a=\sum_i a_i\zeta_i$, the coordinates of
$ca$ are obtained by forming the $d^2$ products $a_ic_j$ and combining them with
the structure constants in each of the $d$ output coordinates. Multiplication by
a fixed element of $S$ therefore costs at most $(0,\,d^3-d,\,d^3+d^2)$. An
$R$-linear form in $d$ coordinates, and in particular $\pi$, costs at most
$(0,\,d-1,\,d)$. Addition in $S$ costs $(0,d,0)$, and halving costs $(0,0,d)$.

\emph{The nonrecursive map over $R$.} At level $j\ge2$, put $m:=2^{j-1}-1$ and
$q:=2m+1$, so that $u,s,v,t$ have $m$ coordinates and each operand of $B_j$ has
$q$. The four scalar--vector products and the two weighted pairings in
\eqref{eq:nonrecursive-map} consist of $6m$ input-dependent multiplications.
Forming the two vector outputs costs $2m$ subtractions and the scalar output
$2m-1$ additions and subtractions. Multiplying by the pairing weights costs at
most $2m$ fixed-scalar multiplications, and multiplying the first vector output
by $\mu_j$ costs $m$ more. One evaluation of $B_j$ over $R$ therefore costs at
most $(6m,\,4m-1,\,3m)$.

\emph{Work per call.} Consider a call to $\AltC_j$ with $j\ge3$. Outside its two
recursive calls, it performs the following work. Its nonrecursive map is
evaluated by \eqref{eq:general-fusion}, which consists of the $\beta_S$
evaluations of $B_j$ over $R$; the input forms, which apply each $\alpha_\ell$
and $\beta_\ell$ to the $q$ coefficients of the corresponding operand, for
$2q\beta_S$ forms in all; and the output reconstruction, which forms each of the
$qd$ output coordinates as a sum of $\beta_S$ fixed multiples of the results.
The call also forms $\rho_j s$ and $\rho_j t$, once for both signs, and
multiplies $Z_+-Z_-$ by $(2\rho_j)^{-1}$, for $3m$ multiplications by fixed
elements of $S$. It halves $Z_++Z_-$, and it performs $8m$ additions in $S$:
$4m$ to form $u\pm\rho_j s$ and $v\pm\rho_j t$, $2m$ to form $Z_+\pm Z_-$, and
$2m$ to combine these with $\mathcal F$ and $\mathcal G$. The scalar $\lambda$
occupies the otherwise zero unit coordinate of the second component and requires
no addition. The costs of this work are as follows.
\begin{center}
\small
\setlength{\tabcolsep}{4pt}
\begin{tabular}{@{}lccc@{}}
\toprule
& $M$ & $A$ & $E$ \\
\midrule
$\beta_S$ evaluations of $B_j$ & $6\beta_Sm$ & $\beta_S(4m-1)$ & $3\beta_Sm$ \\
Input forms & $0$ & $2q\beta_S(d-1)$ & $2q\beta_Sd$ \\
Output reconstruction & $0$ & $qd(\beta_S-1)$ & $qd\beta_S$ \\
Fixed multiplications in $S$ & $0$ & $3(d^3-d)m$ & $3(d^3+d^2)m$ \\
Halving $Z_++Z_-$ & $0$ & $0$ & $dm$ \\
Additions in $S$ & $0$ & $8dm$ & $0$ \\
\midrule
Total & $6\beta_Sm$ & $\kappa_Am+3\beta_S(d-1)-d$ & $\kappa_Em+3d\beta_S$ \\
\bottomrule
\end{tabular}
\end{center}
Since $3\beta_S(d-1)-d\le\kappa_A$, $3d\beta_S\le\kappa_E$, and $m+1=2^{j-1}$,
the work of a call to $\AltC_j$ outside its recursive calls costs at most
\begin{equation}\label{eq:call-bound}
\bigl(6\beta_S(2^{j-1}-1),\;\kappa_A2^{j-1},\;\kappa_E2^{j-1}\bigr)\,.
\end{equation}
Absorbing the constant terms in this way makes the bounds on additions and
fixed-scalar multiplications proportional to $2^{j-1}$, so that every level
contributes the same amount below. A call to $\AltC_2$ performs only the work in
the first three rows, with $m=1$, so \eqref{eq:call-bound} also holds for $j=2$.

\emph{The $R$-valued calls.} Consider a call to $\AltR_j$ with $j\ge2$ and put
$m:=2^{j-1}-1$. For $j\ge3$, first consider the branch that makes two calls to
$\AltC_{j-1}$. Its work outside these calls differs from the table in two ways.
The nonrecursive map is a single evaluation of $B_j$ over $R$, so the first row
has $\beta_S=1$ and the next two rows are absent. Applying $\pi$ to the $2m$
recovered coordinates adds a row $(0,\,2m(d-1),\,2md)$. With the remaining rows
unchanged, this work costs at most
\[
\bigl(6m,\;(3d^3+7d+2)m-1,\;3(d^3+d^2+d+1)m\bigr)\,.
\]
In the branch $\rho_j\in R$, the same operations are performed over $R$ and the
projections are omitted, so this estimate still applies. The one-call branch
uses the same evaluation of $B_j$ and $2m$ linear forms, but only $2m$ fixed
multiplications in $S$, no halving, and $2m$ additions in $S$ and $2m$ in $R$.
These cost no more than the corresponding rows of the table, so the estimate
applies to this branch as well. Since $d\ge1$ and $\beta_S\ge1$, both
$3d^3+7d+2\le\kappa_A$ and $3(d^3+d^2+d+1)\le\kappa_E$, so this is at most
\begin{equation}\label{eq:root-bound}
\bigl(6(2^{j-1}-1),\;\kappa_A2^{j-1},\;\kappa_E2^{j-1}\bigr)\,.
\end{equation}
For $j=2$, the call evaluates only its nonrecursive map over $R$, at cost
$(6,3,3)$, which also satisfies \eqref{eq:root-bound}. Since $\beta_S\ge1$,
\eqref{eq:call-bound} therefore bounds the work outside the recursive calls of
either $\AltC_j$ or $\AltR_j$.

\emph{Counting calls.} Each call at level $j\ge3$ makes at most two calls at
level $j-1$, and calls at level $2$ make none. Hence an evaluation beginning
with $\AltR_n$ has at most $2^{n-j}$ descendant calls at level $j$ for
$2\le j\le n-1$, each to $\AltC_j$ or $\AltR_j$, and there are none for $n=2$.
By \eqref{eq:call-bound}, these calls at each level together cost at most
$\kappa_A2^{n-1}$ additions and $\kappa_E2^{n-1}$ fixed-scalar multiplications,
and by \eqref{eq:root-bound} with $j=n$ the root costs at most the same. Over
the $n-1$ levels from $2$ to $n$, the evaluation of $\AltR_n(X,Y)$ therefore
costs at most $(n-1)\kappa_A2^{n-1}$ additions and $(n-1)\kappa_E2^{n-1}$
fixed-scalar multiplications. For its input-dependent multiplications, the
descendant calls contribute at most
\[
\sum_{j=2}^{n-1}2^{n-j}\cdot6\beta_S(2^{j-1}-1)
=6\beta_S\bigl((n-3)2^{n-1}+2\bigr)
=\beta_S\bigl(3(n-3)2^n+12\bigr)\,,
\]
which also holds for $n=2$, where the sum is empty, and the root contributes at
most $6(2^{n-1}-1)$.

\emph{Reconstruction and weights.} In \eqref{eq:scalar--imaginary}, the product
$\alpha\beta$, the $N-1$ products in the pairing, and the $2(N-1)$ products in
$\alpha Y$ and $\beta X$ are $3N-2$ input-dependent multiplications. Summing the
pairing costs $N-2$ additions and subtracting it from $\alpha\beta$ costs $1$,
while forming $\alpha Y+\beta X$ and adding $\AltR_n(X,Y)$ cost $N-1$ each, for
$3N-3$ in total. The pairing weights add at most $N-1$ fixed-scalar
multiplications, so reconstruction costs at most $(3N-2,\,3N-3,\,N-1)$. Since
$\mathbf w_j=(\mathbf w_{j-1},-\mu_j\mathbf w_{j-1})$, preparing the weights
forms $2^{j-1}$ new entries at each level $j\le n$, each by $1$ fixed-scalar
multiplication and $1$ negation, counted as a subtraction. This costs at most
$(0,\,N-1,\,N-1)$ in total.

Adding the costs of $\AltR_n(X,Y)$, reconstruction, and weight preparation gives
\[
\begin{aligned}
M_{\mathrm{gen}}
&\le\beta_S\bigl(3(n-3)2^n+12\bigr)+6(2^{n-1}-1)+3\cdot2^n-2 \\
&=\bigl(3\beta_S(n-3)+6\bigr)2^n+12\beta_S-8\,, \\
A_{\mathrm{gen}}
&\le(n-1)\kappa_A2^{n-1}+(3N-3)+(N-1)
=\bigl(\kappa_A(n-1)+8\bigr)2^{n-1}-4\,, \\
E_{\mathrm{gen}}
&\le(n-1)\kappa_E2^{n-1}+(N-1)+(N-1)
=\bigl(\kappa_E(n-1)+4\bigr)2^{n-1}-2\,,
\end{aligned}
\]
which are the bounds in \eqref{eq:general-arithmetic-counts}.

\emph{Small $n$.} For $n=1$, the alternating product vanishes, so the
reconstruction omits its addition and costs $(4,2,1)$; with weight preparation
at cost $(0,1,1)$, the total is $(4,3,2)$. For $n=0$, multiplication requires
$1$ product in $R$.

\emph{The one-call root.} The root now makes one call to $\AltC_{n-1}$, so there
are $2^{n-1-j}$ calls to $\AltC_j$ for $2\le j\le n-1$, which contribute at most
\[
\sum_{j=2}^{n-1}2^{n-1-j}\cdot6\beta_S(2^{j-1}-1)
=6\beta_S\bigl((n-3)2^{n-2}+1\bigr)
\]
input-dependent multiplications, with the sum empty for $n=2$. Adding the root's
$6(2^{n-1}-1)$ and the $3N-2$ of reconstruction gives
\eqref{eq:one-call-counts}; for $n=1$ the count is $4$. The bounds on
$A_{\mathrm{gen}}$ and $E_{\mathrm{gen}}$ follow from \eqref{eq:root-bound} and
the smaller number of recursive calls.
\end{proof}

\begin{corollary}[Quasilinear multiplication]\label{cor:quasilinear}
Under the hypotheses of Theorem~\ref{thm:arithmetic}, multiplication in $A_n$
uses $O(n2^n)=O(N\log N)$ arithmetic operations in $R$ for $n\ge1$, where
$N=2^n$.
\end{corollary}

\begin{proof}
The quantities $d$, $\beta_S$, $\kappa_A$, and $\kappa_E$ are independent of
$n$. Adding the three bounds in \eqref{eq:general-arithmetic-counts} therefore
gives $O(n2^n)$ total operations, and $n=\log_2N$ gives the stated bound in $N$.
\end{proof}

\subsection{Bilinear complexity}
\label{subsec:bilinear-complexity}

For a bilinear map $f\colon U_1\times U_2\to W$ between free $R$-modules of
finite rank, its bilinear complexity \cite{BurgisserEtAl1997} is the least
integer $r$ for which
\[
f(x,y)=\sum_{j=1}^{r}\alpha_j(x)\beta_j(y)w_j
\]
with $R$-linear forms $\alpha_j\colon U_1\to R$, $\beta_j\colon U_2\to R$, and
fixed $w_j\in W$. Write $T_{A_n}$ for the multiplication tensor of $A_n$ and
$\Rk_R(T_{A_n})$ for its bilinear complexity, which over a field is the usual
tensor rank.

\begin{corollary}[Bilinear-complexity bound]\label{thm:rank}
For $n\ge1$, the evaluation in Theorem~\ref{thm:arithmetic} gives
\begin{equation}\label{eq:general-rank-bound}
\Rk_R(T_{A_n})
\le\bigl(3\beta_S(n-3)+6\bigr)2^n+12\beta_S-8
=O(N\log N)\,.
\end{equation}
With the one-call root of Corollary~\ref{cor:one-call} for $n\ge3$ and the
direct base cases for $n=1,2$,
\begin{equation}\label{eq:one-call-rank-bound}
\Rk_R(T_{A_n})
\le\bigl(3\beta_S(n-3)+12\bigr)2^{n-1}+6\beta_S-8\,.
\end{equation}
\end{corollary}

\begin{proof}
By the definition of bilinear complexity, it suffices to write
\[
xy=\sum_{j=1}^r \mathcal L_j(x)\mathcal M_j(y)\,w_j
\]
with $R$-linear forms $\mathcal L_j,\mathcal M_j\colon A_n\to R$, fixed
$w_j\in A_n$, and $r$ at most the stated bound. Such an expression is equivalent
to a \emph{normal form}
\[
xy=C(Ax\odot By)\,,
\]
where $A,B\colon A_n\to R^r$ and $C\colon R^r\to A_n$ are $R$-linear and $\odot$
is the coordinatewise product, with $\mathcal L_j$ and $\mathcal M_j$ being the
coordinates of $A$ and $B$, and $w_j=C(e_j)$. The evaluation of
Theorem~\ref{thm:arithmetic} will be put in this form. For $n=1$ the evaluation
consists of the $4$ products in \eqref{eq:scalar--imaginary}, each of a
coordinate of $x$ by a coordinate of $y$. These give the required bilinear
decomposition, so assume $n\ge2$.

Represent coefficients in $S$ by arrays over $R$ as in
Section~\ref{sec:recursion}, and regard $\Mul_n$ as the outermost call, with
operands $x$ and $y$. The input-dependent products of the evaluation, that is,
its multiplications of two input-dependent quantities, are the coordinate
products in each evaluation of a nonrecursive map, by \eqref{eq:general-fusion}
over $S$ or directly over $R$, and the products $\alpha\beta$, $X_jY_j$,
$\alpha Y_j$, and $X_j\beta$ in \eqref{eq:scalar--imaginary}. Every other
operation used in the evaluation is an addition or a multiplication by fixed
data and is therefore $R$-linear on these arrays. Let $r$ be the number of
input-dependent products in the evaluation. The proof of
Theorem~\ref{thm:arithmetic} counts exactly these in bounding
$M_{\mathrm{gen}}$, in \eqref{eq:general-arithmetic-counts} for the two-call
root and in \eqref{eq:one-call-counts} for the one-call root, so $r$ satisfies
the required bound in either case. It remains to verify the following.
\begin{enumerate}
\item[\textup{(1)}] In every call, the first operand is an $R$-linear function
of $x$ alone, and the second is an $R$-linear function of $y$ alone.
\item[\textup{(2)}] Each input-dependent product multiplies an $R$-linear form
in the first operand of its call by an $R$-linear form in the second.
\item[\textup{(3)}] The output of every call is an $R$-linear combination of the
input-dependent products performed during that call, including those in its
recursive calls.
\end{enumerate}
By (1) and (2), the $j$th input-dependent product equals
$\mathcal L_j(x)\mathcal M_j(y)$ for $R$-linear forms $\mathcal L_j$ and
$\mathcal M_j$, and these define $A$ and $B$. Applied to the outermost call, (3)
gives the map $C$.

\emph{Proof of (1).} The operands of $\Mul_n$ are $x$ and $y$, and it passes
$X:=\ImPart(x)$ and $Y:=\ImPart(y)$ to $\AltR_n$; these are $R$-linear in $x$
and $y$ respectively. A call to $\AltR_k$ or $\AltC_k$ with $k\ge3$ and operands
$X=(u,\gamma+s)$ and $Y=(v,\delta+t)$ passes $u\pm\rho_k s$ and $v\pm\rho_k t$
to its recursive calls, or, for a one-call branch, $u+\rho_k s$ and $v+\rho_k t$
to its single call. The first is formed from $X$ alone by selecting coordinates,
multiplying by the fixed coefficient $\rho_k$, and adding, so it is $R$-linear
in $X$; likewise, the second is $R$-linear in $Y$. When $\AltR_k$ calls
$\AltC_{k-1}$, coefficients in $R$ are passed as elements of $S$, and this
inclusion is $R$-linear. Composing these maps from the outermost call downward
proves (1).

\emph{Proof of (2).} The coordinate products in the nonrecursive map
\eqref{eq:nonrecursive-map} are
\[
s_j\delta,\quad \gamma t_j,\quad u_j\delta,\quad \gamma v_j,\quad
s_jv_j,\quad u_jt_j\,,
\]
each a coordinate of the first operand times a coordinate of the second. Over
$R$, these are the input-dependent products themselves. Over $S$,
\eqref{eq:general-fusion} evaluates each through the input-dependent products
$\alpha_\ell(a)\beta_\ell(b)$, where $a$ and $b$ represent the two coordinates
and $\alpha_\ell$ and $\beta_\ell$ are $R$-linear. In
\eqref{eq:scalar--imaginary}, $\alpha$ and $X_j$ are coordinates of $x$, while
$\beta$ and $Y_j$ are coordinates of $y$. This proves (2).

\emph{Proof of (3).} The output of a nonrecursive map is formed from its
input-dependent products by the fixed pairing weights, multiplication by
$\mu_k$, additions, and, over $S$, reconstruction with the fixed $c_\ell$. It is
therefore an $R$-linear combination of these products. Now perform induction on
the level $k$ of a call to either $\AltC_k$ or $\AltR_k$. At level $1$ the
output is zero, and at level $2$ it is the output of the nonrecursive map. For
$k\ge3$, a two-call branch forms its output from $Z_\pm$ and the outputs
$\mathcal F$, $\lambda$, and $\mathcal G$ of the nonrecursive map by additions,
halving, and multiplication by $(2\rho_k)^{-1}$, with $\pi$ applied when the
recursive outputs have coefficients in $S$ and must be recovered over $R$. A
one-call branch applies $\pi_0$ and $\pi_1$ to its recursive output and combines
the results with $\mathcal F$, $\lambda$, and $\mathcal G$. All these operations
are $R$-linear, the nonrecursive outputs are $R$-linear combinations of
input-dependent products, and by the induction hypothesis so are the recursive
outputs. Finally, $\Mul_n$ forms its output from the output of $\AltR_n$ and its
own input-dependent products by additions and multiplication by the fixed
pairing weights. This proves (3) and completes the proof.
\end{proof}

\begin{corollary}[Fields, the real tower, and tensor rank]\label{cor:field-real}
Let $F$ be a field of characteristic different from $2$ and fix
$\mu_k\in F^\times$ for each $k\in\mathbb Z^{>0}$. Suppose that each $\mu_k$,
$k\ge3$, has a square root in a common field extension $K/F$ of finite degree.
Then multiplication in $A_n$ uses $O(N\log N)$ arithmetic operations in $F$ and
satisfies the bilinear bound \eqref{eq:general-rank-bound} with $R=F$ and $S=K$.

Suppose in addition that $K=F(\sqrt\nu)$ is a quadratic extension. Then the
bounds of Theorem~\ref{thm:arithmetic} and \eqref{eq:general-rank-bound} hold
with $d=2$ and $\beta_S=3$. Moreover, Algorithm~\ref{alg:multiplication}, in
which each call to $\AltR_k$ with $k\ge3$ uses the one-call root of
Corollary~\ref{cor:one-call} when $\rho_k\notin F$ and the branch for
$\rho_k\in R$ when $\rho_k\in F$, gives for every $n\ge1$
\begin{equation}\label{eq:field-one-call-bound}
\Rk_F(T_{A_n})\le(9n-15)2^{n-1}+10\,,
\end{equation}
whatever the parameters. Independently of this assumption, if every $\mu_k$ with
$k\ge3$ is a square in $F$, then no extension is required: $S=F$ may be taken,
and
\begin{equation}\label{eq:field-square-bound}
\Rk_F(T_{A_n})\le(3n-3)2^n+4\qquad(n\ge2)\,.
\end{equation}

In particular, for $F=\R$ and any nonzero real parameters, multiplication uses
$O(N\log N)$ real arithmetic operations and satisfies
\eqref{eq:field-one-call-bound}; taking each $\mu_k=-1$ recovers the standard
real tower, and \eqref{eq:field-square-bound} applies when every $\mu_k$ with
$k\ge3$ is positive. If $F=\mathbb F_q$ is a finite field of odd characteristic,
the hypotheses hold with $K=\mathbb F_{q^2}$ for every choice of parameters, so
\eqref{eq:field-one-call-bound} holds over every finite field of odd
characteristic; for the parameters $\mu_k=-1$, \eqref{eq:field-square-bound}
applies exactly when $q\equiv1\pmod4$.
\end{corollary}

\begin{proof}
Choose an $F$-basis of $K$ containing $1$ and let $\pi\colon K\to F$ be the
coefficient projection onto $1$. Then $\pi(1)=1$, and $2\in F^\times$ since $F$
has characteristic different from $2$, so the hypotheses of
Section~\ref{sec:recursion} hold with $R=F$ and $S=K$. The arithmetic bound is
Corollary~\ref{cor:quasilinear} and the bilinear bound is
Corollary~\ref{thm:rank}.

Now let $K=F(\sqrt\nu)$ with $\nu\in F$ a non-square. The basis $\{1,\sqrt\nu\}$
gives $d=2$, and the identity
\[
(a_1+a_2\sqrt\nu)(b_1+b_2\sqrt\nu)
=a_1b_1(1-\sqrt\nu)+a_2b_2(\nu-\sqrt\nu)
 +(a_1+a_2)(b_1+b_2)\sqrt\nu
\]
is a formula of the form \eqref{eq:coefficient-decomposition} with $3$
input-dependent products, so $\beta_S=3$ may be taken.

For \eqref{eq:field-one-call-bound}, let $M_k$ denote the number of
input-dependent $F$-products in one call to $\AltC_k$, including its recursive
calls, and $C_k$ the same number for one call to $\AltR_k$ on inputs with
coefficients in $F$. By \eqref{eq:call-bound} with $\beta_S=3$, a call to
$\AltC_k$ with $k\ge3$ uses at most $18(2^{k-1}-1)$ products outside its two
recursive calls, and $M_2=18$, so $M_k\le2M_{k-1}+18(2^{k-1}-1)$ and by
induction
\[
M_k\le(9k-18)2^k+18\qquad(k\ge2)\,.
\]
It is to be shown that
\begin{equation}\label{eq:real-input-products}
C_n\le(9n-21)2^{n-1}+12\qquad(n\ge2)\,.
\end{equation}
For $n=2$ the call evaluates only $B_2$ over $F$, using $6$ products, which is
\eqref{eq:real-input-products} with equality. Let $n\ge3$ and put
$m:=2^{n-1}-1$. The map $B_n$ over $F$ uses $6m$ products, and the remaining
products are in the recursive calls, which depend on whether $\rho_n$ lies in
$F$.

If $\rho_n\notin F$, then $\{1,\rho_n\}$ is an $F$-basis of $K$, and its
coordinate projections $\pi_0,\pi_1\colon K\to F$ are $F$-linear maps satisfying
$\pi_0(r+\rho_nr')=r$ and $\pi_1(r+\rho_nr')=r'$ for all $r,r'\in F$; they can
be evaluated in the basis $\{1,\sqrt\nu\}$. Thus Corollary~\ref{cor:one-call}
applies, and the one-call root gives
\[
C_n\le M_{n-1}+6m
\le(9n-27)2^{n-1}+18+3\cdot2^n-6
=(9n-21)2^{n-1}+12\,.
\]
If $\rho_n\in F$, the transformed inputs $u\pm\rho_ns$ and $v\pm\rho_nt$ have
coefficients in $F$, so the branch for $\rho_n\in R$ in $\AltR_n$ makes two
calls to $\AltR_{n-1}$, and by induction
\[
\begin{aligned}
C_n&\le2C_{n-1}+6m \\
&\le(9n-30)2^{n-1}+24+3\cdot2^n-6 \\
&=(9n-21)2^{n-1}+12-(3\cdot2^{n-1}-6)\,,
\end{aligned}
\]
which is at most the right-hand side of \eqref{eq:real-input-products}. This
proves the claim. Adding the $3N-2$ products of the reconstruction
\eqref{eq:scalar--imaginary} gives
$(9n-21)2^{n-1}+12+3\cdot2^n-2=(9n-15)2^{n-1}+10$ products for $n\ge2$, and for
$n=1$ the direct formula uses $4$, the value of \eqref{eq:field-one-call-bound}
at $n=1$.

It remains to check that this evaluation is bilinear in the sense of the proof
of Corollary~\ref{thm:rank}, so that its number of input-dependent products
bounds $\Rk_F(T_{A_n})$. Conditions (1)--(3) of that proof hold for the branch
$\rho_n\in F$ exactly as for the others: the operands $u\pm\rho_ns$ and
$v\pm\rho_nt$ are $F$-linear in $X$ and $Y$ separately, the only input-dependent
products are those of $B_n$ and of the two recursive calls, and the output is
formed from $Z_\pm$, $\mathcal F$, $\lambda$ and $\mathcal G$ by additions and
multiplications by the fixed scalars $1/2$ and $(2\rho_n)^{-1}$. This proves
\eqref{eq:field-one-call-bound}.

If every $\mu_k$ with $k\ge3$ is a square in $F$, then $S=F$ satisfies the
hypotheses of Section~\ref{sec:recursion} with $d=1$, $\beta_S=1$ and
$\pi=\mathrm{id}$, and \eqref{eq:general-rank-bound} gives
$\Rk_F(T_{A_n})\le(3n-3)2^n+4$ for $n\ge2$. Equivalently, the branch
$\rho_k\in F$ is taken at every level, and the recurrence
$C_n=2C_{n-1}+6(2^{n-1}-1)$ with $C_2=6$ gives $C_n=(3n-6)2^n+6$.

When $F=\R$, take $K=\C$, write $\iota^2=-1$, and set $\rho_k:=\sqrt{\mu_k}$ if
$\mu_k>0$ and $\rho_k:=\iota\sqrt{-\mu_k}$ if $\mu_k<0$. These roots give the
required coefficient extension for every nonzero real parameter,
$K=\R(\sqrt{-1})$ is quadratic, and $\rho_k\in\R$ exactly when $\mu_k>0$. Taking
$\mu_k=-1$ gives $\rho_k=\iota$ and recovers the standard real tower.

When $F=\mathbb F_q$ with $q$ odd, take $K=\mathbb F_{q^2}$. The group
$K^\times$ is cyclic of order $q^2-1$, so its squares form its unique subgroup
of order $(q^2-1)/2=(q-1)\tfrac{q+1}{2}$. Since $q$ is odd, $q-1$ divides this
order, so the subgroup of squares contains the unique subgroup of $K^\times$ of
order $q-1$, which is $F^\times$. Hence every element of $F^\times$, and in
particular every $\mu_k$, is a square in $K$, and $K=F(\sqrt\nu)$ for any
non-square $\nu\in F$, which exists since $q$ is odd. Finally, $-1$ is a square
in $\mathbb F_q$ if and only if $4$ divides $q-1$, since $\mathbb F_q^\times$ is
cyclic of order $q-1$ and $-1$ is its unique element of order $2$, which lies in
the index-two subgroup of squares exactly when $(q-1)/2$ is even.
\end{proof}

For the real and finite fields above, the octonion base case improves
\eqref{eq:field-one-call-bound} to
\[
\Rk_F(T_{A_n})\le(36n-69)2^{n-3}+10\qquad(n\ge4)\,,
\]
as follows from Theorem~\ref{thm:leaf-rank} of Appendix~\ref{sec:real-schedule}.

\subsection{Auxiliary storage}

It remains to bound the workspace of Algorithm~\ref{alg:multiplication}.

\begin{proposition}[Linear auxiliary storage]\label{prop:workspace}
Algorithm~\ref{alg:multiplication}, with the level-two base case, admits a
sequential evaluation using $O(N)$ auxiliary coefficients of $R$ in addition to
the input and output storage.
\end{proposition}

\begin{proof}
For $n=0$, the evaluation requires $1$ multiplication in $R$, so let $n\ge1$.
Evaluate the algorithm depth-first: each call completes one recursive call
before beginning the next, and releases the arrays it has allocated when it
returns. The \emph{frame} of a call consists of the arrays it allocates,
including the operand and output arrays of its recursive calls but excluding its
own operand and output arrays, which belong to the frame of its caller or, for
$\Mul_n$, to the input and output storage. Subarrays of an existing array, such
as the components $u,s,v,t$ of the operands of $\AltC_k$ or the imaginary parts
of $x$ and $y$, are read in place and allocate nothing.

At any moment, the calls in progress form a chain
\[
\Mul_n,\ \AltR_n,\ \AltC_{n-1},\ \AltC_{n-2},\ \ldots,\ \AltC_j\,,
\]
in which each call after the first is the recursive call currently being
evaluated by the one before it. When $\AltR_n$ takes the branch $\rho_n\in R$,
the chain instead continues with calls to $\AltR_{n-1},\AltR_{n-2},\ldots$ until
a level at which $\rho_k\notin R$, after which it continues with calls to
$\AltC$. Thus, after the first two calls, at most one call, to $\AltR_k$ or
$\AltC_k$, is in progress at each level $k<n$. The auxiliary coefficient storage
at that moment therefore consists of the frames of the calls in this chain and
the shared data stored once for all calls. It suffices to prove the following.
\begin{enumerate}
\item[\textup{(1)}] There is a constant $c$, depending only on the fixed
coefficient representation and formulas, such that the frame of a call to
$\AltC_k$ or $\AltR_k$ has at most $c\,2^k$ coefficients of $R$, and the frame
of $\Mul_n$ has at most $cN$.
\item[\textup{(2)}] The shared data occupy $O(N)$ coefficients of $R$.
\end{enumerate}
Given (1), the frames in the chain occupy at most
\[
cN+cN+\sum_{k=1}^{n-1}c\,2^k<3cN
\]
coefficients of $R$, so together with (2), the auxiliary storage is $O(N)$. Each
call in the chain also records a bounded number of indices and array references
and the point at which it resumes after its current recursive call. This control
information is the usual cost of recursion and amounts to $O(\log N)$ machine
words, separate from the coefficient storage.

\emph{Proof of (1).} Consider a call to $\AltC_k$ with $k\ge3$ and operands
$X=(u,\gamma+s)$ and $Y=(v,\delta+t)$, and put $m:=2^{k-1}-1$ and $q:=2m+1$. The
call proceeds in three stages.

First, it evaluates the nonrecursive map to obtain
$(\mathcal F,\lambda,\mathcal G)=B_k^S(X,Y)$ by summing
\eqref{eq:general-fusion} one term at a time. For each $\ell$, applying
$\alpha_\ell$ and $\beta_\ell$ to the operands produces two arrays of length $q$
over $R$, and evaluating $B_k$ on them produces a third, with its pairing
accumulated in a scalar. This result is multiplied by $c_\ell$ and added to an
accumulator of length $q$ over $S$. The three arrays over $R$ are reused for
every $\ell$, so this stage allocates at most $3q+dq$ coefficients of $R$,
together with a bounded number of scalars for arithmetic in $S$. When it
finishes, the accumulator holds $(\mathcal F,\lambda,\mathcal G)$.

Second, it forms $\rho_k s$ and $\rho_k t$ and evaluates $Z_+$ and then $Z_-$.
The operand arrays $u\pm\rho_k s$ and $v\pm\rho_k t$ and the output arrays
$Z_\pm$ each have length $m$ over $S$. Even with a separate array for each, this
stage allocates at most $8m$ coefficients of $S$, that is, $8dm$ coefficients of
$R$.

Third, it evaluates \eqref{eq:two-call-formula} coordinate by coordinate,
writing directly into its output array with a bounded number of scalars.

Summing the three stages, the frame has at most $(10m+1)d+3q$ coefficients of
$R$ plus a bounded number of scalars, which is at most $c\,2^k$ for a suitable
constant $c$. A call to $\AltC_1$ allocates nothing, and a call to $\AltC_2$
performs only the first stage, so their frames are bounded.

A call to $\AltR_k$ with $k\ge3$ proceeds through the same stages, except that
its nonrecursive map is evaluated directly over $R$ into a single array of
length $q$. In the branch $\rho_k\in R$, the recursive operand and output arrays
are also stored over $R$. In the remaining two-call branch, the operands are
formed in $S$ and $\pi$ is applied coordinatewise during the third stage. The
one-call branch forms only one pair of recursive operands and one recursive
output array, and applies $\pi_0$ and $\pi_1$ during the third stage. For
$k\le2$, the call performs at most the first stage. Thus, enlarging $c$ if
necessary, every call to $\AltR_k$ has a frame of at most $c\,2^k$ coefficients
of $R$.

Finally, $\Mul_n$ reads $\alpha$, $X$, $\beta$, and $Y$ in place, allocates the
output array of $\AltR_n(X,Y)$, of length $N-1$, accumulates the pairing
$\langle X,Y\rangle_n$ in a scalar, and writes \eqref{eq:scalar--imaginary}
directly into its output array. Its frame therefore also has at most $cN$
coefficients of $R$.

\emph{Proof of (2).} The multiplication table for $S$ and the coefficient
formula \eqref{eq:coefficient-decomposition} have size independent of $n$. The
parameters $\mu_k$ for $1\le k\le n$, the representations of $\rho_k$ and
$(2\rho_k)^{-1}$ for $3\le k\le n$, and the values of the projection maps on the
generators occupy $O(dn)$ coefficients of $R$. These maps are $\pi$ and, at each
level using the one-call branch, $\pi_0$ and $\pi_1$. The pairing-weight lists
$\mathbf w_j$ in \eqref{eq:weights}, for $0\le j\le n$, have combined length
\[
\sum_{j=0}^n2^j=2^{n+1}-1=2N-1\,.
\]
Since $n\le N$, the shared data occupy $O(N)$ coefficients of $R$.
\end{proof}

The culmination of these results is the proof of Theorem~\ref{thm:intro-main}.

\begin{proof}[Proof of Theorem~\ref{thm:intro-main}]
Corollary~\ref{cor:quasilinear} gives the total arithmetic bound, and
Proposition~\ref{prop:workspace} gives the sequential evaluation with linear
auxiliary storage. The proof of Corollary~\ref{thm:rank} shows that the
input-dependent products multiply fixed $R$-linear forms in the separate
operands. Their number is $O(N\log N)$ by \eqref{eq:general-arithmetic-counts}.
The case $n=1$ is the direct base formula.
\end{proof}

\section{Implementation and benchmarks}
\label{sec:implementation}

\subsection{The implementation}
\label{subsec:implementation-overview}

The new quasilinear algorithm presented in this paper is implemented in
\texttt{fastCD}, a standalone C11 library for the standard Cayley--Dickson tower
over $\R$, $\C$, and every finite field $\mathbb{F}_q$ with $q=p^k$, where
$p\le2^{31}-1$ is an odd prime and $k\le12$. The accompanying \texttt{pyfastcd}
extension makes these same routines available in Python through NumPy arrays and
is available on PyPI at
\href{https://pypi.org/project/pyfastcd}{https://pypi.org/project/pyfastcd}.
Source code, tests, benchmark programs, reproduction scripts, and archived data
are distributed under the MIT license at
\url{https://git.tanuki-cd.com/algebraity/fastCD}, with a full mirror of the
repository at \url{https://github.com/TANUKI-CD/fastCD} \cite{LemleyFastCD2026}.
The main timing benchmarks presented in this paper were run on a prerelease
snapshot of \texttt{fastCD}~v1.5.0 at revision
\href{https://git.tanuki-cd.com/algebraity/fastCD/commit/93c9bd8ab1c3bcfdb16c76e5ad55fa487d741a18}{\texttt{93c9bd8ab1}},
whose C library sources differ from those of the v1.5.1 release only
in the version string. The JSON data, build records, and Markdown reports for
the timing benchmarks are preserved in the
\href{https://github.com/TANUKI-CD/fastCD/tree/v1.5.1/benchmark-results/t440-20261120}{benchmark archive}.

In the implementation, elements are stored as contiguous arrays of coefficients
in the recursive standard basis, where a coefficient is a binary64 (C
\texttt{double}) value over $\R$, a pair of binary64 values over $\C$, and a
list of $k$ residues modulo $p$ over $\mathbb F_{p^k}$, representing a
polynomial reduced modulo a fixed monic irreducible polynomial of degree $k$.
Over $\R$, the quasilinear algorithm is evaluated using the real arithmetic
schedule described in Appendix~\ref{sec:real-schedule}. Table~\ref{tab:domains}
lists the coefficient algebra $S$ and how the alternating product is evaluated
in each domain. Over $\C$, the implementation calls $\AltC_n$ directly, so no
coefficient extension is required. Over $\mathbb F_q$, the schedule has the same
form as over $\R$, with $S=\mathbb F_q[\iota]/(\iota^2+1)$ and the one-call
root, and with the level-two base case used in place of the octonion leaf.

\begin{table}[tbp]
\centering
\caption{Domains implemented in \texttt{fastCD}, with the coefficient algebra
$S$ and the evaluation of the alternating product used by the quasilinear
schedule, the methods available, and the form of the correctness check. CC
denotes Cariow--Cariowa and QL denotes quasilinear multiplication.}
\label{tab:domains}
\small
\setlength{\tabcolsep}{5pt}
\begin{tabular}{@{}lllll@{}}
\toprule
Domain & $S$ & Evaluation & Methods & Check \\
\midrule
$\R$ & $\C$ & one-call root & direct, CC, QL & tolerance \\
$\C$ & $\C$ & direct call to $\AltC_n$ & direct, QL & tolerance \\
$\mathbb F_q$ & $\mathbb F_q[\iota]/(\iota^2+1)$ & one-call root
  & direct, QL & equality \\
\bottomrule
\end{tabular}
\end{table}

In addition to this new algorithm, the library includes an optimized
implementation of the defining recursion in every domain and, over $\R$, a
complete matrix-free implementation of the uniform Cariow--Cariowa method,
supported up to dimension $N=1024$ \cite{CariowCariowa2015}. All of these
methods include specialized kernels for small dimensions and can be used for
both single and batched products, with the choice of algorithm either specified
by the user or made automatically through the \texttt{AUTO} setting. Finally,
the library also exposes functions for computing fused sums of products of real
Cayley--Dickson elements and the covariance between samples of them, which is
used for instance in computing the $Q_{2^n}$ metric in hyperspectral image
quality assessment \cite{GarzelliNencini2009}; benchmarks and numerical
stability results for these features are included with the code, but they are
not discussed in this paper. More information on the implementation and examples
of its use in C and Python are available in the repository documentation.

Over $\R$ and $\C$, every sum in the quasilinear schedule whose length grows
with $N$, namely the scalar output of each nonrecursive map and the pairing in
\eqref{eq:scalar--imaginary}, is accumulated in blocks of $8$ terms, with the
block totals combined by a balanced tree rather than by a sequential sum. This
changes only the order in which the terms of each sum are added, so the
operation counts of Theorem~\ref{thm:real-counts} are unaffected. Above
dimension $256$, the Cariow--Cariowa correction sums also combine their block
totals by a balanced tree. The specialized kernels are produced by a generator
distributed with the library, and the kernels in the repository are checked
against its output.

In order to compare the algorithms and choose the default algorithm for each
dimension in each domain, benchmarks were performed on one core of an Intel Xeon
Gold 6148 processor (nominal frequency $2.40$ GHz), running Linux with code
compiled using GCC~16.2.1 and the flags \texttt{-O3 -DNDEBUG -fno-fast-math
-ffp-contract=off}, the last two of which disable value-changing floating-point
optimizations and the contraction of products and sums into fused operations. At
each dimension $N=2,4,\ldots,1024$, all methods available in a domain were
tested on the same randomly generated input pairs, with $10{,}000$ pairs over
$\R$ and $1{,}000$ pairs over $\C$ and over each finite field, where direct
multiplication and the verification of the results are considerably more
expensive. Both single and batched multiplication were measured, with one native
C call made for each input pair in the single case and one for the complete set
of pairs in the batch case.

All storage was allocated before timing began, so that the measurements include
the cost of the C calls but exclude input generation, verification of the
results, and Python object overhead. The benchmarks used $24$ blocks with a
rotated task order, run as separate subprocesses over $\R$ and within a single
process for each field and dimension otherwise, with separate stack offsets for
the blocks in the latter case. Each timing observation consisted of two warmup
passes and then repeated passes over the input pairs, calibrated to a target of
$20$ milliseconds. For every algorithm, the results on all input pairs were
checked against an independent implementation of the defining recursion, and all
results passed the benchmark's checks. Over $\R$, agreement required both the
computed product $\widehat p$ and the independently evaluated binary64 reference
$p_{\mathrm{ref}}$ to have finite coordinates and satisfy
\[
\|\widehat p-p_{\mathrm{ref}}\|_\infty\le128N\epsilon\|x\|_2\|y\|_2\,,
\]
where
$\epsilon=2^{-52}$ is the binary64 machine epsilon. Here
\[
\|z\|_2=(\sum_j|z_j|^2)^{1/2}
\]
is the Euclidean norm of the standard
coefficient vector, and
\[
\|z\|_\infty:=\max_j|z_j|\,.
\]
Over $\C$, the reference
was evaluated in extended precision (C \texttt{long double}), and agreement was
required up to a norm-scaled tolerance. Over $\mathbb F_q$, where arithmetic is
exact, the reference used schoolbook polynomial arithmetic, and agreement
required equality. The full procedure, build information, raw data, confidence
intervals, and supporting metadata are included in the reports.

In every domain, quasilinear multiplication had the lowest mean time among the
methods tested at every sufficiently large tested dimension, namely at every
$N\ge32$ over $\R$, at every $N\ge8$ over $\C$, and at every $N\ge8$ over each
of the nine tested finite fields, apart from single products over
$\mathbb F_{3^5}$ at $N=64$. The results for the three domains are presented in
Sections~\ref{subsec:bench-real}, \ref{subsec:bench-complex},
and~\ref{subsec:bench-finite}. Finally, the results of experiments measuring
floating-point error of the quasilinear implementation over $\R$ and $\C$ are
presented in Section~\ref{subsec:stability}.

\subsection{Real coefficients}
\label{subsec:bench-real}

Over $\R$, all three methods were tested on input pairs generated with operand
norms in $[0.5,2)$. Batched products were evaluated with SIMD for the
quasilinear and Cariow--Cariowa kernels at $16\le N\le512$ and for the direct
kernel at $64\le N\le512$, with the single-product kernels used for other values
of $N$.

Table~\ref{tab:multiplication-benchmark} gives the mean time per product in
nanoseconds, averaged over the $24$ subprocess blocks, with the direct column
referring to the optimized implementation of the defining recursion. The
quasilinear algorithm has the lowest mean time among the three methods at every
tested dimension $N\ge32$ for both single and batched products, and at $N=1024$
it is approximately $17$ times faster than direct multiplication. The report
also includes $95\%$ confidence intervals for these means and for the paired
timing ratios.

\begin{table}[tbp]
\centering
\caption{Mean native multiplication times over $\R$ in nanoseconds per product
over $24$ subprocess blocks, rounded to two decimal places. CC denotes
Cariow--Cariowa and QL denotes quasilinear multiplication.}
\label{tab:multiplication-benchmark}
\small
\setlength{\tabcolsep}{3pt}
\begin{tabular}{@{}r rrr rrr@{}}
\toprule
& \multicolumn{3}{c}{Batch products}
& \multicolumn{3}{c}{Single products} \\
\cmidrule(lr){2-4}\cmidrule(lr){5-7}
$N$ & Direct & CC & QL & Direct & CC & QL \\
\midrule
2    &      1.32 &      1.52 &      1.02 &     15.76 &     16.76 &     15.30 \\
4    &      3.84 &      4.65 &      3.68 &     17.75 &     20.22 &     18.08 \\
8    &     12.18 &     14.05 &     17.59 &     28.22 &     38.16 &     36.25 \\
16   &     54.79 &     49.92 &     52.41 &     76.28 &    110.28 &    108.25 \\
32   &    256.56 &    227.38 &    184.50 &    298.82 &    432.14 &    221.13 \\
64   &   1133.91 &    973.51 &    465.99 &   1163.54 &   1835.74 &    518.34 \\
128  &   4530.78 &   5006.53 &   1242.81 &   4495.22 &   7910.79 &   1333.57 \\
256  &  17665.04 &  17695.42 &   3046.46 &  17445.55 &  29214.67 &   3244.06 \\
512  &  69956.80 &  68578.92 &   6920.47 &  69310.72 & 108032.81 &   7427.95 \\
1024 & 275135.85 & 433524.98 &  16209.21 & 277013.17 & 419988.45 &  16258.80 \\
\bottomrule
\end{tabular}
\end{table}

\FloatBarrier

Although the quasilinear algorithm has the best asymptotic complexity of the
three methods, this does not mean that it is the fastest at every lower
dimension, and the benchmark results show why the other methods remain useful.
At $N=8$, direct multiplication has the lowest mean time for both single and
batched products, while at $N=16$, Cariow--Cariowa is fastest for batches and
direct multiplication is fastest for single products, as it also is at $N=4$.
Based on the operation counts of Theorem~\ref{thm:real-counts} and
Proposition~\ref{prop:direct-cc-counts} at $N=1024$,
$T_{\mathrm{direct}}/T_{\mathrm{QL}} \approx18.0$, close to the measured batch
timing ratio of approximately $17.0$. This agreement does not hold for all three
implementations. Indeed, by operation count alone, direct multiplication
requires about $1.9$ times the operations of Cariow--Cariowa at $N=1024$, yet
the Cariow--Cariowa implementation was measured to be about $1.6$ times slower.
Thus the total operation counts in Table~\ref{tab:operation-counts} do not by
themselves determine running time on this platform. The uniform Cariow--Cariowa
schedule uses fewer real multiplications at several small dimensions than the
uniform quasilinear schedule counted in Table~\ref{tab:operation-counts}, so a
different implementation of it with less overhead could change these results.
The results reported here therefore describe the performance of matrix-free
implementations after attempts to optimize each of them and should not be
interpreted as showing that Cariow--Cariowa is generally slower.

Based on these measurements, the \texttt{AUTO} setting over $\R$ selects
different algorithms at different dimensions, with separate choices for single
and batched products. For single products, direct multiplication is used at
$N=4$, $N=8$, and $N=16$, with quasilinear multiplication used otherwise, while
for batched products direct multiplication is used at $N=8$, Cariow--Cariowa at
$N=16$, and quasilinear multiplication otherwise. In builds without the explicit
SIMD kernels, batched multiplication uses the same algorithm selection as single
multiplication for simplicity and portability. These choices are intended to
give the best performance for each dimension without requiring users to choose
an algorithm for each dimension themselves, while still allowing any of the
three algorithms to be selected explicitly when it is more suitable for a
particular application or platform.

\subsection{Complex coefficients}
\label{subsec:bench-complex}

Over $\C$, the library implements the standard tower with complex coefficients,
$A_n\otimes_\R\C$. Since $\iota\in\C$, the alternating product of two elements
of $V_n^\C$ is computed by applying the kernel $\AltC_n$ of the real quasilinear
schedule directly at level $n$ without the one-call root and full multiplication
is recovered by \eqref{eq:scalar--imaginary} over $\C$. By
\eqref{eq:kernel-counts} with $k=n$, the alternating product uses
$(36n-81)2^{n-2}+18$ input-dependent real multiplications for $n\ge3$, in the
counting convention of Appendix~\ref{subsec:real-statements}, and the
reconstruction adds $3N-2$ products of complex coefficients, whereas direct
evaluation of \eqref{eq:CD-product} uses $N^2$ such products.

The input pairs were generated with complex coefficients whose real and
imaginary parts are independent standard normal samples, and direct and
quasilinear multiplication were tested. Table~\ref{tab:complex-benchmark} gives
the mean time per product in nanoseconds, averaged over the $24$ blocks. The
quasilinear algorithm has the lowest mean time at every tested dimension $N\ge8$
for both single and batched products, and at $N=1024$ it is approximately $44$
times faster than direct multiplication, while direct multiplication has the
lower mean time at $N=2$ and $N=4$. Counting each product of complex
coefficients as $4$ real multiplications, the counts above give $83{,}722$ real
multiplications for quasilinear multiplication and $4{,}194{,}304$ for direct
multiplication at $N=1024$, a ratio of approximately $50.1$, compared with the
measured batch timing ratio of approximately $43.7$.

\begin{table}[tbp]
\centering
\caption{Mean native multiplication times over $\C$ in nanoseconds per product
over $24$ blocks, rounded to two decimal places. QL denotes quasilinear
multiplication.}
\label{tab:complex-benchmark}
\small
\setlength{\tabcolsep}{3pt}
\begin{tabular}{@{}r rr rr@{}}
\toprule
& \multicolumn{2}{c}{Batch products}
& \multicolumn{2}{c}{Single products} \\
\cmidrule(lr){2-3}\cmidrule(lr){4-5}
$N$ & Direct & QL & Direct & QL \\
\midrule
2    &      12.59 &      19.10 &      36.75 &      43.82 \\
4    &      33.24 &      34.54 &      58.15 &      59.51 \\
8    &     126.44 &     103.86 &     150.62 &     129.87 \\
16   &     513.75 &     286.48 &     538.36 &     312.88 \\
32   &    2093.11 &     701.63 &    2136.38 &     732.72 \\
64   &    8402.27 &    1654.11 &    8416.90 &    1704.32 \\
128  &   34235.33 &    3875.53 &   34183.72 &    3920.68 \\
256  &  132844.26 &    8755.40 &  132701.42 &    8701.93 \\
512  &  535402.27 &   20359.74 &  534252.09 &   20555.33 \\
1024 & 2137253.87 &   48958.79 & 2136829.06 &   48752.39 \\
\bottomrule
\end{tabular}
\end{table}

\FloatBarrier

Based on these measurements, the \texttt{AUTO} setting over $\C$ selects direct
multiplication at $N=2$ and $N=4$ and quasilinear multiplication otherwise, for
both single and batched products.

\subsection{Finite fields of odd characteristic}
\label{subsec:bench-finite}

Over a finite field $\mathbb F_q$ of odd characteristic, the library implements
the standard tower, with every doubling parameter equal to $-1$. The quasilinear
schedule has the same form as over $\R$, with the coefficient algebra
$S=\mathbb F_q[\iota]/(\iota^2+1)$, the formula \eqref{eq:real-fusion}, and the
one-call root. When $q\equiv3\pmod4$, this algebra is the field
$\mathbb F_{q^2}$ of Corollary~\ref{cor:field-real}, and when $q\equiv1\pmod4$
it is not a field, but it satisfies the hypotheses of
Section~\ref{sec:recursion} in either case, so the same schedule is used for
every odd $q$. Arithmetic in $\mathbb F_q$ is exact, so the measurements in this
domain concern running time only.

Arithmetic in $\mathbb F_q$ is performed by kernels developed for the library,
which are specialized by the extension degree $k$ and, on x86 processors, use
SSE4.1, AVX2, or AVX-512 instructions selected at run time. From $N=16$ onward,
direct multiplication uses the same field arithmetic as the quasilinear kernels,
so that both algorithms are measured with the same coefficient arithmetic at
these dimensions. In order to assess this arithmetic separately from the
multiplication algorithms, it was also compared with FLINT~3.6.0 \cite{FLINT},
an independent C library for exact arithmetic.

The benchmarks used the prime fields $\mathbb F_3$, $\mathbb F_5$,
$\mathbb F_{13}$, $\mathbb F_{65521}$, and $\mathbb F_{2^{31}-1}$ and the
extension fields $\mathbb F_{3^5}$, $\mathbb F_{5^2}$, $\mathbb F_{13^5}$, and
$\mathbb F_{65521^2}$. The input pairs were generated with independent uniformly
distributed residues, and direct and quasilinear multiplication were tested. The
comparison with FLINT used these nine fields together with $\mathbb F_{3^{12}}$,
$\mathbb F_{13^{12}}$, and $\mathbb F_{(2^{31}-1)^{12}}$.

The comparison with FLINT consists of two workloads with identical operands and
defining polynomials. The first is a single multiplication in $\mathbb F_q$, for
which FLINT's \texttt{nmod\_mul} is used over prime fields and
\texttt{fq\_nmod\_mul} over extension fields. The second is a complete product
in $A_n$, in which both sides evaluate the same quasilinear schedule and only
the coefficient arithmetic differs. Table~\ref{tab:flint-ratios} gives the ratio
of the time per product with FLINT to the time with the library's own arithmetic
for both workloads. For a single multiplication in $\mathbb F_q$, the library's
arithmetic is approximately $2$ to $7$ times faster across the twelve fields,
and for complete products the ratio is between $2.1$ and $36.6$ at $N=16$ and
between $49$ and $465$ at $N=1024$. The ratios for complete products also
include the specialization of the library's kernels by degree and differences in
data layout, so they do not isolate the cost of a multiplication in
$\mathbb F_q$. The purpose of this test is not to prove a general speedup over
FLINT, but rather to motivate the choice to use a custom kernel for finite field
arithmetic in \texttt{fastCD} instead of relying on FLINT.

\begin{table}[tbp]
\centering
\caption{Ratios of the time per product with FLINT~3.6.0 arithmetic to the time
with the arithmetic of \texttt{fastCD}, as geometric means of the paired ratios
over $24$ blocks, for a single multiplication in $\mathbb F_q$ and for complete
quasilinear products in $A_n$ at four dimensions $N$. Values above $1$ favor
\texttt{fastCD}.}
\label{tab:flint-ratios}
\small
\setlength{\tabcolsep}{5pt}
\begin{tabular}{@{}l r rrrr@{}}
\toprule
& & \multicolumn{4}{c}{Complete products} \\
\cmidrule(l){3-6}
Field & $\mathbb F_q$ & $N=16$ & $N=64$ & $N=256$ & $N=1024$ \\
\midrule
$\mathbb F_{3}$ & 3.50 & 7.6 & 22.2 & 116 & 243 \\
$\mathbb F_{5}$ & 3.54 & 8.27 & 23.4 & 92.2 & 150 \\
$\mathbb F_{13}$ & 3.53 & 8.68 & 25.6 & 76.7 & 120 \\
$\mathbb F_{65521}$ & 3.55 & 8.5 & 24.6 & 44.6 & 49.2 \\
$\mathbb F_{2^{31}-1}$ & 3.49 & 8.75 & 24.7 & 42.5 & 48.6 \\
$\mathbb F_{3^5}$ & 2.86 & 36.6 & 164 & 288 & 465 \\
$\mathbb F_{5^2}$ & 6.22 & 31.5 & 72.9 & 111 & 112 \\
$\mathbb F_{13^5}$ & 4.20 & 14.8 & 83.7 & 110 & 122 \\
$\mathbb F_{65521^2}$ & 7.02 & 15 & 46.1 & 67.1 & 73.4 \\
$\mathbb F_{3^{12}}$ & 2.36 & 25.9 & 183 & 291 & 370 \\
$\mathbb F_{13^{12}}$ & 1.96 & 2.11 & 21.3 & 54.3 & 77.6 \\
$\mathbb F_{(2^{31}-1)^{12}}$ & 3.69 & 16.3 & 58 & 98.5 & 113 \\
\bottomrule
\end{tabular}
\end{table}

Table~\ref{tab:finite-benchmark} gives the mean time per single product in
nanoseconds for quasilinear multiplication over the nine fields, together with
the paired ratio of the direct time to the quasilinear time, at four of the
tested dimensions. From $N=16$ onward, the batch and single means differ by at
most $8\%$, and the times and ratios for both modes at every tested dimension
are included in the reports. Direct multiplication has the lower mean at $N=2$
over every field and at $N=4$ over all five prime fields. The $95\%$ confidence
interval of the paired ratio contains $1$ for single products over
$\mathbb F_{65521^2}$ at $N=2$, both modes over $\mathbb F_{3^5}$ at $N=4$ and
$N=64$, batch products over $\mathbb F_{13^5}$ at $N=4$, and batch products over
$\mathbb F_{65521^2}$ at $N=16$. From $N=8$ onward, the quasilinear algorithm
has the lower mean in every case except single products over $\mathbb F_{3^5}$
at $N=64$. At $N=1024$, the quasilinear algorithm is between $8.4$ and $71$
times faster than direct multiplication.

\begin{table}[tbp]
\centering
\caption{Mean native quasilinear multiplication times over $\mathbb F_q$ in
nanoseconds per single product over $24$ blocks, rounded to two decimal places,
each followed by the paired ratio of the direct time to the quasilinear time.}
\label{tab:finite-benchmark}
\small
\setlength{\tabcolsep}{3pt}
\begin{tabular}{@{}l rr rr rr rr@{}}
\toprule
& \multicolumn{2}{c}{$N=16$}
& \multicolumn{2}{c}{$N=64$}
& \multicolumn{2}{c}{$N=256$}
& \multicolumn{2}{c}{$N=1024$} \\
\cmidrule(lr){2-3}\cmidrule(lr){4-5}\cmidrule(lr){6-7}\cmidrule(l){8-9}
Field & QL & Ratio & QL & Ratio & QL & Ratio & QL & Ratio \\
\midrule
$\mathbb F_{3}$       & 303.08 & 1.53 & 729.53 & 2.38 & 950.01 & 2.76 &
2202.39 & 17.50 \\
$\mathbb F_{5}$       & 302.11 & 1.52 & 836.99 & 2.21 & 1222.39 & 6.59 &
3955.03 & 33.08 \\
$\mathbb F_{13}$      & 303.82 & 1.53 & 844.72 & 2.19 & 1479.45 & 7.26 &
5001.40 & 35.05 \\
$\mathbb F_{65521}$   & 299.13 & 1.54 & 804.65 & 2.27 & 2942.70 & 10.15 &
12742.09 & 36.58 \\
$\mathbb F_{2^{31}-1}$ & 298.39 & 1.50 & 849.23 & 2.20 & 2698.65 & 10.41 &
12466.56 & 37.68 \\
$\mathbb F_{3^5}$     & 1053.90 & 1.08 & 1715.12 & 0.98 & 5265.21 & 2.23 &
18651.26 & 8.58 \\
$\mathbb F_{5^2}$     & 570.05 & 1.27 & 1885.78 & 4.17 & 7405.65 & 15.61 &
35498.85 & 50.35 \\
$\mathbb F_{13^5}$    & 3443.56 & 1.20 & 4611.17 & 6.18 & 19779.29 & 21.79 &
96222.81 & 70.77 \\
$\mathbb F_{65521^2}$ & 1231.36 & 1.09 & 2828.12 & 3.94 & 11629.44 & 14.30 &
56283.80 & 45.93 \\
\bottomrule
\end{tabular}
\end{table}

\FloatBarrier

The operation counts in Table~\ref{tab:operation-counts} are not directly
applicable in this domain, since the finite-field implementation uses the
level-two base case instead of the octonion leaf. Moreover, the costs of
polynomial multiplication and modular reduction over $\mathbb F_q$ depend on the
characteristic and the degree. Thus, the improvement over direct multiplication
varies between fields, so the timing ratios over $\mathbb F_q$ are reported
without a corresponding comparison of operation counts. Moreover, since all
blocks for a given field and dimension were run within a single process, the
confidence intervals in the reports do not include variation between processes,
and in repeated runs on the same machine individual times at $N\ge64$ differed
by up to $17\%$ and ratios of direct to quasilinear times by up to $8\%$.

Based on these measurements, the \texttt{AUTO} setting over $\mathbb F_q$
selects direct multiplication at $N=2$ and $N=4$ and quasilinear multiplication
otherwise, for both single and batched products.

\subsection{Measured floating-point error}
\label{subsec:stability}

The products computed over $\R$ and $\C$ are affected by rounding. Since the
quasilinear schedule transforms its operands and recombines halved sums and
differences of the recursive results, its accuracy is not evident from that of
the defining recursion. The accuracy of the implemented methods was therefore
measured in a study whose results are preserved in the
\href{https://github.com/TANUKI-CD/fastCD/tree/v1.5.1/benchmark-results/stability-20261120}{stability archive}.
The study computes every reference product exactly, in integer arithmetic on the
binary64 inputs, and reports for nonzero $x,y$
\[
\eta:=\frac{\|\widehat{xy}-xy\|_2}{\mathsf u\,\|x\|_2\|y\|_2}\,,
\qquad\mathsf u:=\epsilon/2=2^{-53}\,,
\]
where $\widehat{xy}$ is the computed product and $\mathsf u$ is the unit
roundoff of binary64 arithmetic, so that $\eta$ is the error divided by the
scale $\mathsf u\|x\|_2\|y\|_2$. The study covers $6{,}976$ test cases over $\R$
at $N=2,4,\ldots,2048$ and $6{,}608$ cases over $\C$ at $N=2,4,\ldots,512$,
drawn from $21$ real and $19$ complex families that include Gaussian and uniform
coordinates, sparse and correlated operands, operands constructed to produce
cancellation, and perturbations of zero divisors. Each pair is multiplied by the
quasilinear schedule as implemented, by the same schedule with sequential sums
in place of the blocked accumulation described in
Section~\ref{subsec:implementation-overview}, and by direct multiplication, and
the real pairs with $N\le1024$ are also multiplied by the Cariow--Cariowa
method. The balanced-accumulation measurements use the same revision as the main
timing benchmarks, \texttt{93c9bd8ab1}, and the sequential-accumulation baseline
uses revision
\href{https://git.tanuki-cd.com/algebraity/fastCD/commit/982b9daed40e61e59fd939dae61a0e358a35191a}{\texttt{982b9daed4}}.

Table~\ref{tab:measured-error} gives the mean and the maximum of $\eta$ over the
pairs with Gaussian and uniform coordinates. Over $\R$, the mean of $\eta$ for
the quasilinear schedule rises from $1.54$ at $N=16$ to $2.91$ at $N=2048$,
against $1.29$ to $1.80$ for direct multiplication, and a least-squares fit of
the squared means for $4\le n\le11$ gives approximately $0.85\,n-0.95$ and
$0.22\,n+0.86$, respectively. Over $\C$, the corresponding fits for $4\le n\le9$
are $0.82\,n-0.16$ and $0.21\,n+0.40$. Over the tested range, the average error
is consistent with $\sqrt{\log N}$ growth for both methods, and that of the
quasilinear schedule is less than $1.8$ times that of direct multiplication at
every tested dimension. A possible explanation is that rounding errors at
successive levels accumulate like independent perturbations, as in probabilistic
models of rounding error \cite{HighamMary2019}. The number of pairs decreases
with the dimension, so the averages at the largest dimensions rest on few
samples.

\begin{table}[tbp]
\centering
\caption{Mean and maximum of $\eta$ over the input pairs with Gaussian and
uniform coordinates, for quasilinear (QL) and direct multiplication, with the
number of pairs at each dimension in each domain.}
\label{tab:measured-error}
\small
\setlength{\tabcolsep}{4pt}
\begin{tabular}{@{}rr rr rr@{\hspace{12pt}}rr rr@{}}
\toprule
& & \multicolumn{4}{c}{Real coefficients}
& \multicolumn{4}{c}{Complex coefficients} \\
\cmidrule(lr){3-6}\cmidrule(l){7-10}
& & \multicolumn{2}{c}{QL} & \multicolumn{2}{c}{Direct}
& \multicolumn{2}{c}{QL} & \multicolumn{2}{c}{Direct} \\
\cmidrule(lr){3-4}\cmidrule(lr){5-6}\cmidrule(lr){7-8}\cmidrule(l){9-10}
$N$ & Pairs & Mean & Max & Mean & Max & Mean & Max & Mean & Max \\
\midrule
2 & 512 & 0.50 & 1.51 & 0.50 & 1.51 & 0.71 & 1.75 & 0.71 & 1.75 \\
4 & 512 & 0.69 & 1.81 & 0.74 & 1.93 & 0.96 & 3.51 & 0.88 & 1.93 \\
8 & 512 & 1.03 & 2.65 & 1.00 & 2.21 & 1.51 & 3.20 & 1.00 & 1.80 \\
16 & 512 & 1.54 & 3.69 & 1.29 & 2.38 & 1.75 & 3.55 & 1.11 & 1.66 \\
32 & 512 & 1.82 & 3.70 & 1.42 & 2.48 & 1.98 & 3.91 & 1.21 & 1.62 \\
64 & 192 & 2.06 & 2.90 & 1.49 & 1.88 & 2.22 & 2.89 & 1.30 & 1.54 \\
128 & 192 & 2.25 & 2.93 & 1.56 & 1.95 & 2.37 & 3.05 & 1.37 & 1.54 \\
256 & 64 & 2.41 & 2.80 & 1.62 & 1.80 & 2.53 & 2.85 & 1.44 & 1.55 \\
512 & 64 & 2.61 & 2.90 & 1.70 & 1.82 & 2.68 & 2.90 & 1.52 & 1.59 \\
1024 & 16 & 2.73 & 2.94 & 1.74 & 1.81 &  &  &  &  \\
2048 & 16 & 2.91 & 3.03 & 1.80 & 1.84 &  &  &  &  \\
\bottomrule
\end{tabular}
\end{table}

Over all families and dimensions, the largest value of $\eta$ observed for the
quasilinear schedule is $3.70$ over $\R$ and $4.15$ over $\C$, against $2.94$
and $2.18$ for direct multiplication. For the Cariow--Cariowa method, the mean
of $\eta$ on the pairs with Gaussian and uniform coordinates is at most $3.82$,
but the largest value observed over all families is $79.7$, at $N=1024$. These
figures describe the tested inputs and do not necessarily say anything about
worst-case bounds: the study contains no inputs constructed to maximize the
rounding error, and no bound on the error of any of the methods is proved in
this paper.

The measurements also show the effect of the blocked accumulation. With
sequential sums, the largest value of $\eta$ for the quasilinear schedule rises
to $78.3$ over $\R$, at $N=2048$, and to $18.3$ over $\C$, at $N=512$, and on
the family of real inputs constructed so that long sums cancel, the median of
$\eta$ at $N=512$ is $16.0$, against $1.08$ with the blocked accumulation. In
the timing measurements of the same study, the blocked accumulation costs at
most $5\%$ for real products at the tested dimensions $N\ge64$ with the default
kernels.

Finally, the product-relative error is
\[
\frac{\|\widehat{xy}-xy\|_2}{\|xy\|_2}
=\mathsf u\,\eta\,\frac{\|x\|_2\|y\|_2}{\|xy\|_2}\,,
\qquad(xy\ne0)\,.
\]
The factor $\|x\|_2\|y\|_2/\|xy\|_2$ is unbounded for $n\ge4$ because $A_n$ then
has zero divisors, so the relative error of a product near a zero divisor can be
large while $\eta$ remains small. For the $8$ cases with perturbations of size
$2^{-52}$ of the zero divisors 
\[
(e_1+e_{10})(e_4-e_{15})=0
\]
at $N=512$, the
quasilinear schedule has median $\eta\approx1.27$, within the range of
Table~\ref{tab:measured-error}, while its median relative error
\[
\|\widehat{xy}-xy\|_2/\|xy\|_2
\]
is about $0.64$. Direct multiplication is
nearly exact on the same inputs, with a median relative error of about
$10^{-16}$. Applications that depend on the relative accuracy of products near
zero divisors should take this difference into account. On all tested inputs,
the measured errors are far below the tolerance
\[
128N\epsilon\|x\|_2\|y\|_2=256N\mathsf u\|x\|_2\|y\|_2
\]
for real inputs and
\[
64N\epsilon\|x\|_2\|y\|_2=128N\mathsf u\|x\|_2\|y\|_2
\]
for complex inputs used
for the correctness checks of Section~\ref{subsec:implementation-overview}.

\section{Outlook and future work}
\label{sec:outlook}

The two-call alternating-product recursion presented here gives a uniform
quasilinear method for multiplication throughout the Cayley--Dickson tower under
the hypotheses of Theorem~\ref{thm:intro-main}, together with an explicit
$O(N\log N)$ bilinear upper bound and a linear-workspace evaluation. No
optimality claim for the overall multiplication algorithm or matching lower
bound is presented. It remains natural to ask whether the logarithmic overhead
can be reduced and whether further specialized low-dimensional base cases can
reduce the arithmetic constants. Furthermore, while the algorithm presented in
this paper is fairly general, applying over a unital commutative ring in which
$2$ is invertible and allowing a large range of doubling parameters, the
question of whether a similar algorithm to that presented in this paper applies
when $2$ is not a unit (e.g. in fields of characteristic $2$) may be worth
pursuing. It may also be fruitful to investigate whether a more efficient
algorithm is possible over specific fields or rings with more structure to
exploit.

In addition to these theoretical questions, several practical matters are
likewise left open. A rigorous worst-case error bound in addition to the
empirical forward-error analysis would clarify the numerical behavior of the
transformed recursion, while cache-aware scheduling and accelerator
implementations could change the observed crossover points. Broader benchmark
campaigns are needed before drawing conclusions across hardware or input
distributions. Finally, integrating the present kernel into applications that
already use higher Cayley--Dickson arithmetic would test whether its isolated
multiplication gains translate into end-to-end improvements.

\section{Statements and disclosures}
\label{sec:disclosures}

\subsection{Use of generative AI}

Generative AI is a substantial part of the author's research process. The author
made use of frontier models through OpenAI's ChatGPT and Codex and Anthropic's
Claude and Claude Code in the development of this paper.

These models were used to assist with literature search and review, identify
open problems, suggest research directions, explain and analyze papers, assist
with drafting and editing, discuss proof ideas, suggest lemmas and consequences,
outline proofs, and develop computational tests of results. They were also used
substantively during the development of proofs, in some cases proposing key
arguments or generating candidate proofs in response to mathematical problems
posed by the author.

\subsection{Code availability}

The code and data discussed in this paper and used to support its results are
publicly available under the MIT license at
\href{https://git.tanuki-cd.com/algebraity/fastCD}{https://git.tanuki-cd.com/algebraity/fastCD},
with a GitHub mirror available at
\href{https://github.com/TANUKI-CD/fastCD}{https://github.com/TANUKI-CD/fastCD}.

\subsection{Competing interests and funding}

The author is the developer of the \texttt{fastCD} library benchmarked in
Section~\ref{sec:implementation} and declares no other competing interests. This
research received no specific funding.

\subsection{Statement of responsibility}

All results presented in this paper have been independently verified by the
author. The author takes full responsibility for the contents of the paper,
including its mathematical claims, arguments, citations, and reported
computations.

\section*{Acknowledgements}
The author would like to thank Professor Per Bäck of Mälardalen University for
the suggestion to generalize the quasilinear algorithm for the real
Cayley--Dickson tower to arbitrary fields, which has been further generalized to
the class of unital commutative rings satisfying the hypotheses of the present
paper.

\appendix

\section{Real base case and exact counts}
\label{sec:real-schedule}

This appendix fixes the arithmetic schedule used by \texttt{fastCD} for the
standard real Cayley--Dickson tower and establishes its exact operation counts.
Throughout, $R=\R$, $S=\C$, $\mu_k=-1$ for every $k$, and $\rho_k=\iota$ with
$\iota^2=-1$, except in the identity of Theorem~\ref{thm:octonion-leaf}, which
is stated over a general field, and in Theorem~\ref{thm:leaf-rank}, which
concerns the bilinear complexity over the fields of
Corollary~\ref{cor:field-real}. The schedule differs from the general evaluation
counted in Section~\ref{sec:complexity} in one respect: the recursive kernel
uses an explicit formula for the alternating product on the complexified
imaginary octonions as its base case, which lowers the linear constant of the
multiplication count.

The base-case identity and the operation counts are checked by exact
computation. The script \nolinkurl{verify_octonion_leaf.py}, located in the
\texttt{tools} directory of the repository \cite{LemleyFastCD2026}, at revision
\href{https://git.tanuki-cd.com/algebraity/fastCD/commit/81b625551510ad13b3ae3862e9c4f7d5c0782f01}{\texttt{81b6255155}},
verifies the identity and the operation counts of the base case and checks the
closed forms of Theorem~\ref{thm:real-counts} against their recurrences. The
instrumented counters \nolinkurl{tools/check_quasilinear_counts.py} and
\nolinkurl{tools/check_cariow_counts.py} of the same repository record the
arithmetic counts of the C kernels under the conventions below at every
dimension $N\le1024$.

\subsection{Conventions}
\label{subsec:real-statements}

Write $M$ for input-dependent real multiplications, $A$ for real additions and
subtractions, and $E$ for multiplications by fixed real scalars other than $0$
and $\pm1$, and put $T:=M+A+E$ for the total number of operations. In this
model, coordinate extraction, copying, permutations, and sign changes are free
and are not counted. Multiplication by $\iota$ exchanges the real and imaginary
coordinates of a complex coefficient and changes one sign, so it is also free,
whereas halvings are counted in $E$. Known-zero terms are omitted, sums are
initialized from their first term, and the preprocessing of both variable
operands is included. Scalar and imaginary parts are read from the standard
coordinates. This convention differs from the unit-cost model of
Section~\ref{sec:complexity}, which charges negations and multiplications by
$\pm1$, so the counts below are not directly comparable with the bounds of
Theorem~\ref{thm:arithmetic}.

\subsection{The octonion base case}
\label{subsec:leaf}

The base case concerns level $3$. Write an element of $V_3^\C$ in standard
coordinates as
\[
X=(x_1,x_2,x_3,h,y_1,y_2,y_3)\,,
\]
so that $x_1,x_2,x_3$ are the coordinates of $e_1,e_2,e_3$, $h$ is the
coordinate of $e_4$, and $y_1,y_2,y_3$ are the coordinates of $e_5,e_6,e_7$. For
$u,v\in\C^3$, let
\[
u\times v:=(u_2v_3-u_3v_2,\;u_3v_1-u_1v_3,\;u_1v_2-u_2v_1)
\]
be the ordinary three-dimensional cross product, extended bilinearly to complex
coordinates. Define the linear forms
\[
p:=-\iota h\,,\qquad z:=x_1+p\,,\qquad
b_\pm:=x_2\pm\iota y_2\,,\qquad c_\pm:=x_3\pm\iota y_3
\]
in $X$, and let $L_j\colon\C^7\to\C^3$, $1\le j\le4$, be the linear maps
\begin{equation}\label{eq:leaf-maps}
\begin{aligned}
L_1X&=(x_1,\;y_1,\;p)\,, &
L_2X&=(b_++c_-,\;b_--c_+,\;p)\,, \\
L_3X&=(z+\iota y_1,\;b_+,\;c_+)\,, &
L_4X&=(z-\iota y_1,\;b_-,\;c_-)\,.
\end{aligned}
\end{equation}
Transposes of these maps below are ordinary matrix transposes, without complex
conjugation. The identity in the following theorem is formal, so it is stated
over any field of characteristic other than $2$ containing a square root of
$-1$.

\begin{theorem}[Octonion leaf]\label{thm:octonion-leaf}
Let $K$ be a field of characteristic different from $2$ containing an element
$\iota$ with $\iota^2=-1$, let $A_3(K)$ be the level-three Cayley--Dickson
algebra over $K$ with $\mu_1=\mu_2=\mu_3=-1$, and let $\AltC_3$ be the
alternating product $\times_3$ on its imaginary subspace $V_3(K)$, with the maps
$L_j$ of \eqref{eq:leaf-maps} read over $K$. Then, for all $X,Y\in V_3(K)$,
\begin{equation}\label{eq:octonion-leaf}
\AltC_3(X,Y)
=-\iota\,L_1^{\mathsf T}(L_1X\times L_1Y)
 +\frac12\sum_{j=2}^{4}L_j^{\mathsf T}(L_jX\times L_jY)\,.
\end{equation}
For $K=\C$, regarded as the coefficient algebra of the standard real tower,
formula \eqref{eq:octonion-leaf} has an evaluation in real arithmetic using $72$
input-dependent multiplications, $156$ additions and subtractions, and $15$
halvings, for a total of $243$ operations.
\end{theorem}

The kernel \texttt{ql\_complex\_cross\_8} of \texttt{fastCD}, a straight-line
program emitted by the generator \nolinkurl{tools/generate_kernels.py}, is the
evaluation in Theorem~\ref{thm:octonion-leaf}. It has three stages. The first
applies the maps $L_j$ to both operands, using $9$ complex additions per
operand, since computing $p$ is free. The second evaluates the four complex
cross products $L_jX\times L_jY$, each by three real cross products: for
$u=u_0+\iota u_1$ and $v=v_0+\iota v_1$ with real $u_0,u_1,v_0,v_1$,
\[
\begin{gathered}
u\times v=(\Delta-\Omega)+\iota(W-\Delta-\Omega)\,, \\
\Delta:=u_0\times v_0\,,\qquad
\Omega:=u_1\times v_1\,,\qquad
W:=(u_0+u_1)\times(v_0+v_1)\,.
\end{gathered}
\]
The sum of the real and imaginary parts of $p$ is formed once per operand and
shared by $L_1$ and $L_2$, and the third coordinates of the first two cross
products, which enter only the output coordinate of $e_4$, are combined there
without being reconstructed separately. The third stage collects the entries of
the sparse transposes in \eqref{eq:octonion-leaf}. The three stages have the
following costs.
\begin{center}
\small
\setlength{\tabcolsep}{6pt}
\begin{tabular}{@{}lccc@{}}
\toprule
& $M$ & $A$ & $E$ \\
\midrule
Input linear maps & $0$ & $36$ & $0$ \\
Four complex cross products & $72$ & $88$ & $0$ \\
Output reconstruction & $0$ & $32$ & $15$ \\
\midrule
Total & $72$ & $156$ & $15$ \\
\bottomrule
\end{tabular}
\end{center}
With the one-call root, this base case first occurs in full multiplication in
$A_4$, of real dimension $16$.

\begin{proof}[Proof of Theorem~\ref{thm:octonion-leaf}]
Both sides of \eqref{eq:octonion-leaf} are $K$-bilinear, and in standard
coordinates their structure constants are expressions in $\iota$ and $\tfrac12$
with integer coefficients that do not depend on $K$. Since the ring
$\mathbb Z[\tfrac12][T]/(T^2+1)$ embeds in $\mathbb Q(\iota)$ and maps to $K$ by
$T\mapsto\iota$, it suffices to verify \eqref{eq:octonion-leaf} over
$\mathbb Q(\iota)$ on the $49$ pairs of standard basis vectors, with the
left-hand side computed from the defining product \eqref{eq:CD-product}. The
script \nolinkurl{verify_octonion_leaf.py} performs this verification in exact
rational arithmetic.

For the operation count, the script executes the three stages above, and
separately the generated C kernel itself, on a number type that records every
real operation and admits only sums of like terms and products of a linear form
in $X$ with a linear form in $Y$. Each of the two programs is therefore
real-bilinear by construction. The script confirms that each agrees with the
complex-bilinear extension of $\times_3$ on all $196$ pairs of real basis
vectors of the real and imaginary coefficient planes, so that it computes
$\AltC_3$ on $V_3^\C$, and that every execution uses exactly the numbers of
operations stated, with the division into stages shown in the table.
\end{proof}

\subsection{Operation counts}
\label{subsec:counts}

The \emph{real quasilinear schedule} is Algorithm~\ref{alg:multiplication} with
$R=\R$, $S=\C$, and $\rho_k=\iota$, specialized as follows. The nonrecursive map
$B_k^\C$ is evaluated by \eqref{eq:general-fusion} with the three-product
formula from the proof of Corollary~\ref{cor:field-real}, which for $\nu=-1$
reads
\begin{equation}\label{eq:real-fusion}
B_k^\C(a_0+\iota a_1,\,b_0+\iota b_1)
=(\Delta-\Omega)+\iota(W-\Delta-\Omega)\,,
\end{equation}
where $\Delta:=B_k(a_0,b_0)$, $\Omega:=B_k(a_1,b_1)$, and
$W:=B_k(a_0+a_1,b_0+b_1)$ for operands $a_0+\iota a_1$ and $b_0+\iota b_1$ with
$a_0,a_1,b_0,b_1$ real. The kernel $\AltC_k$ uses \eqref{eq:octonion-leaf} as
its base case at $k=3$, the level-two formula at $k=2$, and the two-call formula
\eqref{eq:two-call-formula} for $k\ge4$. The real-input root $\AltR_n$ uses the
one-call root of Corollary~\ref{cor:one-call} with $\pi_0=\ReCoeff$ and
$\pi_1=\ImCoeff$ for $n\ge3$, where $\ReCoeff$ and $\ImCoeff$ denote the real
and imaginary parts of a complex coefficient; since $\iota\notin\R$, the branch
for $\rho_n\in R$ is never taken. Full multiplication is recovered by
\eqref{eq:scalar--imaginary}.

\begin{theorem}[Exact operation counts]\label{thm:real-counts}
For $n\ge4$ and $N=2^n$, the real quasilinear schedule multiplies two elements
of the standard real algebra $A_n$ using
\begin{equation}\label{eq:real-counts}
\begin{aligned}
M_{\mathrm{QL}}&=\frac{(36n-69)N}{8}+10\,, &
A_{\mathrm{QL}}&=\frac{(76n-196)N}{8}+26\,, \\
E_{\mathrm{QL}}&=\frac{(16n-53)N}{16}+4\,, &
T_{\mathrm{QL}}&=\frac{(240n-583)N}{16}+40\,.
\end{aligned}
\end{equation}
For $n=1,2,3$, the respective counts $(M,A,E)$ are $(4,2,0)$, $(16,12,0)$, and
$(58,62,0)$.
\end{theorem}

\begin{proof}
Let $M_k$, $P_k$, and $E_k$ denote the numbers of input-dependent real products,
real additions and subtractions, and real halvings in one evaluation of the
complex kernel $\AltC_k$, including its recursive calls.
Theorem~\ref{thm:octonion-leaf} gives $(M_3,P_3,E_3)=(72,156,15)$. For $k\ge4$,
put $m:=2^{k-1}-1$ and $q:=2m+1$. Outside its two recursive calls, a call to
$\AltC_k$ evaluates the nonrecursive map by \eqref{eq:real-fusion}, using three
real evaluations of $B_k$ with $6m$ products and $4m-1$ additions each, $2q$
operand sums for $W$, and $3q$ additions for the complex outputs. It forms the
transformed inputs $u\pm\iota s$ and $v\pm\iota t$, forms and halves
$Z_+\pm Z_-$, and adds the results to the two vector outputs of the nonrecursive
map; multiplication by $(2\rho_k)^{-1}=-\iota/2$ is a halving followed by a free
coordinate exchange. The costs are as follows.
\begin{center}
\small
\setlength{\tabcolsep}{6pt}
\begin{tabular}{@{}lccc@{}}
\toprule
& $M$ & $A$ & $E$ \\
\midrule
Nonrecursive map & $18m$ & $22m+2$ & $0$ \\
Transformed inputs & $0$ & $8m$ & $0$ \\
Recursive sums and differences & $0$ & $4m$ & $0$ \\
Vector recombination & $0$ & $4m$ & $0$ \\
Halving & $0$ & $0$ & $4m$ \\
\midrule
Total & $18m$ & $38m+2$ & $4m$ \\
\bottomrule
\end{tabular}
\end{center}
Hence, for $k\ge4$,
\[
\begin{aligned}
M_k&=2M_{k-1}+18(2^{k-1}-1)\,, \\
P_k&=2P_{k-1}+38(2^{k-1}-1)+2\,, \\
E_k&=2E_{k-1}+4(2^{k-1}-1)\,,
\end{aligned}
\]
and induction on $k$ gives
\begin{equation}\label{eq:kernel-counts}
\begin{aligned}
M_k&=(36k-81)2^{k-2}+18\,, \\
P_k&=(76k-168)2^{k-2}+36\,, \\
E_k&=(16k-37)2^{k-3}+4\,,
\end{aligned}
\qquad(k\ge3)\,.
\end{equation}
For $n\ge4$, the one-call root forms $u+\iota s$ and $v+\iota t$ by interleaving
real coordinates, at no arithmetic cost, and makes one call to $\AltC_{n-1}$.
With $m:=2^{n-1}-1$, the real nonrecursive map $B_n$ and the combination of its
two vector outputs with $\ReCoeff Z$ and $\ImCoeff Z$ add $6m$ multiplications
and $6m-1=3N-7$ additions, and the reconstruction \eqref{eq:scalar--imaginary}
adds $3N-2$ multiplications and $3N-3$ additions, with no fixed-scalar
multiplications because the standard pairing weights are all $1$. Therefore
\[
M_{\mathrm{QL}}=M_{n-1}+6N-8\,,\qquad
A_{\mathrm{QL}}=P_{n-1}+6N-10\,,\qquad
E_{\mathrm{QL}}=E_{n-1}\,,
\]
and substituting $k=n-1$ in \eqref{eq:kernel-counts} gives
\eqref{eq:real-counts}. For $n\le3$ the kernel is not called above level two,
and the counts follow by enumerating \eqref{eq:scalar--imaginary}, the level-two
formula, and, for $n=3$, the one-call root with one call to $\AltC_2$ at cost
$(18,24,0)$.

The script \nolinkurl{verify_octonion_leaf.py} checks \eqref{eq:kernel-counts}
and \eqref{eq:real-counts} against these recurrences. Independently, the counter
\nolinkurl{tools/check_quasilinear_counts.py} instruments the C kernels and
confirms \eqref{eq:kernel-counts}, \eqref{eq:real-counts}, and the counts for
$n\le3$ under these conventions at every dimension $N\le1024$.
\end{proof}

The two methods against which the schedule is compared are direct evaluation of
the defining recursion \eqref{eq:CD-product} and the uniform Cariow--Cariowa
method \cite{CariowCariowa2015}, in the matrix-free form implemented in
\texttt{fastCD}: a Walsh--Hadamard evaluation of the unsigned convolution of the
coordinate vectors, followed by a correction, in each output coordinate, of the
terms whose basis-product signs differ from that of the convolution.

\begin{proposition}[Direct and Cariow--Cariowa
counts]\label{prop:direct-cc-counts} Direct evaluation of \eqref{eq:CD-product}
in the standard real tower uses $M_{\mathrm{direct}}=N^2$,
$A_{\mathrm{direct}}=N^2-N$, and $E_{\mathrm{direct}}=0$. For $4\le N\le1024$,
the matrix-free Cariow--Cariowa schedule of \texttt{fastCD} uses
\[
M_{\mathrm{CC}}=\frac{N(N-1)}{2}+2\,,\qquad
A_{\mathrm{CC}}=3nN+\frac{N(N-3)+4}{2}\,,\qquad
E_{\mathrm{CC}}=2N\,,
\]
and $(M,A,E)=(3,7,3)$ at $N=2$.
\end{proposition}

\begin{proof}
Each Cayley--Dickson step in \eqref{eq:CD-product} evaluates four products of
dimension $N/2$ and combines their coordinates using $N$ additions and
subtractions, while the conjugations and the multiplication by $\mu_n=-1$ are
sign changes. Hence
\[
M_{\mathrm{direct}}(n)=4M_{\mathrm{direct}}(n-1)\,,\qquad
A_{\mathrm{direct}}(n)=4A_{\mathrm{direct}}(n-1)+2^n\,,
\]
with $M_{\mathrm{direct}}(0)=1$ and $A_{\mathrm{direct}}(0)=0$, which gives
$N^2$ and $N^2-N$, and $E_{\mathrm{direct}}=0$ because the only fixed
coefficients are signs. The Cariow--Cariowa counts are recorded under these
conventions by the instrumented counter \nolinkurl{tools/check_cariow_counts.py}
at every dimension $2\le N\le1024$. The multiplication count agrees with Cariow
and Cariowa \cite[p.~38]{CariowCariowa2015}; the published count omits the
normalizations and doublings counted in $E_{\mathrm{CC}}$ because they are
power-of-two scalings.
\end{proof}

Table~\ref{tab:operation-counts} lists the resulting counts for $N\le1024$. The
column $E$ is omitted for readability and equals $T-M-A$: it is $0$ for direct
multiplication, $3$ at $N=2$ and $2N$ for $N\ge4$ for Cariow--Cariowa, and $0$
for the quasilinear schedule when $N\le8$ and $(16n-53)N/16+4$ when $N\ge16$.
For $N\le8$, the quasilinear entries for $M$ equal the bound
\eqref{eq:field-one-call-bound} of Corollary~\ref{cor:field-real}, and for
$N\ge16$ the octonion base case reduces the multiplication count below this
bound.

\begin{table}[htbp]
\centering
\caption{Real arithmetic counts for the specified schedules: input-dependent
multiplications $(M)$, additions and subtractions $(A)$, and total operations
$(T=M+A+E)$.}
\label{tab:operation-counts}
\small
\setlength{\tabcolsep}{2.5pt}
\begin{tabular}{@{}r rrr@{\hspace{9pt}}rrr@{\hspace{9pt}}rrr@{}}
\toprule
& \multicolumn{3}{c}{Direct}
& \multicolumn{3}{c}{Cariow--Cariowa}
& \multicolumn{3}{c}{Quasilinear} \\
\cmidrule(lr){2-4}\cmidrule(lr){5-7}\cmidrule(lr){8-10}
$N$ & $M$ & $A$ & $T$ & $M$ & $A$ & $T$ & $M$ & $A$ & $T$ \\
\midrule
2 & 4 & 2 & 6 & 3 & 7 & 13
  & 4 & 2 & 6 \\
4 & 16 & 12 & 28 & 8 & 28 & 44
  & 16 & 12 & 28 \\
8 & 64 & 56 & 120 & 30 & 94 & 140
  & 58 & 62 & 120 \\
16 & 256 & 240 & 496 & 122 & 298 & 452
  & 160 & 242 & 417 \\
32 & 1024 & 992 & 2016 & 498 & 946 & 1508
  & 454 & 762 & 1274 \\
64 & 4096 & 4032 & 8128 & 2018 & 3106 & 5252
  & 1186 & 2106 & 3468 \\
128 & 16384 & 16256 & 32640 & 8130 & 10690 & 19076
  & 2938 & 5402 & 8816 \\
256 & 65536 & 65280 & 130816 & 32642 & 38530 & 71684
  & 7018 & 13210 & 21432 \\
512 & 262144 & 261632 & 523776 & 130818 & 144130 & 275972
  & 16330 & 31258 & 50504 \\
1024 & 1048576 & 1047552 & 2096128 & 523778 & 553474 & 1079300
  & 37258 & 72218 & 116328 \\
\bottomrule
\end{tabular}
\end{table}

\FloatBarrier

\begin{corollary}[Comparison of the schedules]\label{cor:comparisons}
In the standard real tower and for $2\le N\le1024$, the quasilinear schedule
uses fewer input-dependent multiplications than direct multiplication when
$N\ge8$ and than the Cariow--Cariowa schedule when $N\ge32$. Its total operation
count equals that of direct multiplication for $N=2,4,8$ and is smaller for
$N\ge16$, and it is smaller than that of the Cariow--Cariowa schedule at every
such dimension. Moreover, $M_{\mathrm{CC}}+A_{\mathrm{CC}}>T_{\mathrm{QL}}$ at
every such dimension, so the last comparison also holds when fixed-scalar
multiplications are not charged.
\end{corollary}

\begin{proof}
These are comparisons of the entries of Table~\ref{tab:operation-counts}, with
$M_{\mathrm{CC}}+A_{\mathrm{CC}}=T_{\mathrm{CC}}-E_{\mathrm{CC}}$ for the last
assertion.
\end{proof}

\subsection{The bilinear bound over fields}
\label{subsec:leaf-rank}

The octonion base case also improves the bilinear bounds of
Corollary~\ref{cor:field-real}. Its identity is formal, and fixed linear
coordinate changes are free in the bilinear model, so it transports to arbitrary
parameters.

\begin{theorem}[Leaf-improved bilinear bound]\label{thm:leaf-rank}
Let $F$ be a field of characteristic different from $2$, fix $\mu_k\in F^\times$
for each $k\in\mathbb Z^{>0}$, and let $K=F(\sqrt\nu)$ be a quadratic extension
containing an element $\iota$ with $\iota^2=-1$, square roots of $-\mu_1$,
$-\mu_2$, and $-\mu_3$, and a square root $\rho_k$ of $\mu_k$ for every $k\ge3$.
Then, for $n\ge4$,
\begin{equation}\label{eq:leaf-rank-bound}
\Rk_F(T_{A_n})\le(36n-69)2^{n-3}+10\,.
\end{equation}
The hypothesis holds for $F=\R$ with $K=\C$ and any nonzero real parameters, and
for every finite field $F=\mathbb F_q$ of odd characteristic with
$K=\mathbb F_{q^2}$ and any parameters.
\end{theorem}

\begin{remark}\label{rem:leaf-rank}
The bound \eqref{eq:leaf-rank-bound} is \eqref{eq:field-one-call-bound} of
Corollary~\ref{cor:field-real} reduced by $9\cdot2^{n-3}$. The octonion leaf
uses $72$ products where the level-two base case uses $90$, and for the standard
real tower this saving of $18$ products occurs on each of the $2^{n-4}$
level-three calls below the one-call root. For the standard real tower,
\eqref{eq:leaf-rank-bound} equals $M_{\mathrm{QL}}$ in \eqref{eq:real-counts}.
For $n\le3$ the leaf is not reached, and Corollary~\ref{cor:field-real} gives
$4$, $16$, and $58$. When every $\mu_k$ with $k\ge3$ is a square in $F$, the
schedule never enters the kernel over $K$, so it does not improve the bound
\eqref{eq:field-square-bound}.
\end{remark}

For the proof, fix $F$, the parameters, and $K$ as in
Theorem~\ref{thm:leaf-rank}. The first step transports the octonion formula to
the given parameters.

\begin{lemma}[Transport of the base case]\label{lem:transport}
For $1\le k\le3$, choose $\tau_k\in K^\times$ with $\tau_k^2=-\mu_k$, let $D_0$
be the identity on $K$, and put
\[
D_k(a,b):=(D_{k-1}a,\;\tau_kD_{k-1}b)\,.
\]
Then $D_k$ is an isomorphism of $K$-algebras from $A_k\otimes_FK$, with the
given parameters, onto the level-$k$ Cayley--Dickson algebra over $K$ with
parameters $-1$, and it commutes with the involutions. It is diagonal in the
standard coordinates, as is its inverse, and it maps imaginary subspaces onto
imaginary subspaces.
\end{lemma}

\begin{proof}
Suppose that $D_{k-1}$ is such an isomorphism and commutes with the involution,
and write $\tilde z:=D_{k-1}z$ for $z=a,b,c,d$. The involution is $K$-linear, so
it fixes $\tau_k$. Multiplication of the images in the algebra with parameter
$-1$ gives
\[
(\tilde a,\tau_k\tilde b)(\tilde c,\tau_k\tilde d)
=\bigl(\tilde a\tilde c-\tau_k^2\,\bar{\tilde d}\tilde b,\;
       \tau_k(\tilde d\tilde a+\tilde b\bar{\tilde c})\bigr)
=D_k\bigl(ac+\mu_k\bar d b,\;da+b\bar c\bigr)\,,
\]
where the last equality uses $\tau_k^2=-\mu_k$ and the induction hypothesis.
This is the image of the product in \eqref{eq:CD-product}. Moreover,
\[
\overline{D_k(a,b)}=(D_{k-1}\bar a,\;-\tau_kD_{k-1}b)=D_k(\bar a,-b)\,.
\]
Since $\tau_k\ne0$ and $D_{k-1}$ is invertible, $D_k$ is invertible, and its
inverse is diagonal in the standard coordinates because $D_k$ is. The base case
$D_0$ is the identity, and the remaining assertions follow.
\end{proof}

\begin{proof}[Proof of Theorem~\ref{thm:leaf-rank}]
Write $\AltC_3^{(\mu)}$ for the $K$-bilinear alternating product at level $3$
with the given parameters, on $V_3\otimes_FK$, and $\AltC_3^{\mathrm{std}}$ for
the one with parameters $-1$, on $V_3(K)$ as in Theorem~\ref{thm:octonion-leaf}.
Since $D_3$ preserves multiplication and the involution,
Lemma~\ref{lem:transport} gives
\[
\AltC_3^{(\mu)}(X,Y)=D_3^{-1}\AltC_3^{\mathrm{std}}(D_3X,D_3Y)\,,
\]
and $\AltC_3^{\mathrm{std}}$ is given by \eqref{eq:octonion-leaf} over $K$, by
Theorem~\ref{thm:octonion-leaf}. The schedule to be counted is
Algorithm~\ref{alg:multiplication} with $R=F$ and $S=K$, in which each call to
$\AltR_k$ with $k\ge3$ uses the one-call root of Corollary~\ref{cor:one-call}
when $\rho_k\notin F$ and the branch for $\rho_k\in R$ when $\rho_k\in F$, and
the kernel uses the transported formula as its specified base case at level $3$,
with coefficients in $K$ represented in the basis $\{1,\sqrt\nu\}$ and
multiplied by the three-product formula from the proof of
Corollary~\ref{cor:field-real}.

\emph{The base case.} In \eqref{eq:octonion-leaf}, the maps $L_j$, $D_3$,
$D_3^{-1}$, the transposes, and the final $K$-linear combinations are fixed
$F$-linear maps of the coordinate arrays, so the only input-dependent products
are the $6$ coordinate products of each of the $4$ cross products in $K$. Each
product in $K$ is $3$ input-dependent products in $F$, so the base case uses
$72$ products in $F$. Each of these multiplies an $F$-linear form in $L_jD_3X$
by an $F$-linear form in $L_jD_3Y$, hence an $F$-linear form in $X$ by one in
$Y$, and the output is an $F$-linear combination of these products.

\emph{The kernel.} For $k\ge4$, the work of a call to $\AltC_k$ outside its
recursive calls is as in the proof of Corollary~\ref{cor:field-real}: the
parameters and pairing weights contribute only fixed linear operations, and the
nonrecursive map uses $18(2^{k-1}-1)$ products in $F$ by \eqref{eq:call-bound}
with $\beta_S=3$. With $M_k$ denoting the number of products in $F$ in one call
to $\AltC_k$, including its recursive calls, $M_k\le2M_{k-1}+18(2^{k-1}-1)$ for
$k\ge4$ and $M_3=72$ by the preceding paragraph, so by induction
\[
M_k\le(36k-81)2^{k-2}+18\qquad(k\ge3)\,,
\]
in place of the bound $(9k-18)2^k+18$ obtained there from the level-two base
case. Conditions (1)--(3) in the proof of Corollary~\ref{thm:rank} hold for the
kernel by the same induction as before, with the base case replaced by the
property just established at level $3$.

\emph{Calls on inputs over $F$.} Let $C_n$ denote the number of products in $F$
in one call to $\AltR_n$ on inputs with coefficients in $F$, and put
\[
U_n:=(36n-93)2^{n-3}+12\qquad(n\ge4)\,.
\]
It is to be shown that $C_n\le U_n$ for $n\ge4$. A call to $\AltR_3$ makes no
call to the kernel at level $3$, so it is evaluated exactly as in the proof of
Corollary~\ref{cor:field-real}, and \eqref{eq:real-input-products} gives
$C_3\le36$ in either branch. At level $4$, the nonrecursive map uses
$6\cdot7=42$ products. If $\rho_4\notin F$, then $\{1,\rho_4\}$ is an $F$-basis
of $K$, the one-call root applies as in the proof of
Corollary~\ref{cor:field-real}, and $C_4\le M_3+42=114$; if $\rho_4\in F$, the
two calls to $\AltR_3$ give $C_4\le2C_3+42\le114$. Since $U_4=114$, the bound
holds at level $4$. Now let $n\ge5$ and suppose that $C_{n-1}\le U_{n-1}$; the
nonrecursive map at level $n$ uses $6(2^{n-1}-1)$ products. If $\rho_n\notin F$,
the one-call root applies in the same way, and
\[
\begin{aligned}
C_n&\le M_{n-1}+6(2^{n-1}-1) \\
&\le(36n-117)2^{n-3}+18+24\cdot2^{n-3}-6 \\
&=(36n-93)2^{n-3}+12=U_n\,.
\end{aligned}
\]
If $\rho_n\in F$, the two calls to $\AltR_{n-1}$ and the induction hypothesis
give
\[
C_n\le2U_{n-1}+6(2^{n-1}-1)
=(36n-105)2^{n-3}+18
=U_n-(12\cdot2^{n-3}-6)\le U_n\,.
\]
This proves that $C_n\le U_n$ for every $n\ge4$ and every choice of parameters.
Conditions (1)--(3) in the proof of Corollary~\ref{thm:rank} hold for $\AltR_n$
exactly as in the proof of Corollary~\ref{cor:field-real}, since the transformed
inputs are $F$-linear in the separate operands and the maps $\pi_0$ and $\pi_1$,
the halvings, and the multiplications by $(2\rho_n)^{-1}$ are $F$-linear.

\emph{Reconstruction.} Formula \eqref{eq:scalar--imaginary} adds $3N-2$
input-dependent products, each of a coordinate of $x$ by a coordinate of $y$,
and combines them with the output of $\AltR_n$ by additions and the fixed
pairing weights. By the proof of Corollary~\ref{thm:rank}, the evaluation
therefore gives a decomposition of $xy$ into at most $C_n+3N-2$ products of
$F$-linear forms in $x$ and in $y$ with fixed coefficients in $A_n$, and for
$n\ge4$,
\[
C_n+3N-2\le(36n-93)2^{n-3}+12+24\cdot2^{n-3}-2=(36n-69)2^{n-3}+10\,,
\]
which proves \eqref{eq:leaf-rank-bound}.

\emph{The real and finite fields.} For $F=\R$ and $K=\C=\R(\sqrt{-1})$, the
element $\iota$ and the roots $\tau_k$ and $\rho_k$ exist for every nonzero real
parameter. For $F=\mathbb F_q$ with $q$ odd, $K=\mathbb F_{q^2}$ is a quadratic
extension of $F$ in which every element of $F$ is a square, by the proof of
Corollary~\ref{cor:field-real}, which supplies $\iota$, the $\tau_k$, and the
$\rho_k$ for every choice of parameters.
\end{proof}

\printbibliography

\end{document}